\documentclass{article}

\usepackage{amssymb}
\usepackage{amsthm}
\usepackage{amsmath}
\usepackage{mathtools}
\usepackage{float}
\usepackage{tikz}
\usetikzlibrary{arrows.meta}
\usetikzlibrary{decorations.markings}

\usepackage{subcaption}
\usepackage[percent]{overpic}

\usepackage{lipsum}
\usepackage{amsfonts}
\usepackage{epstopdf}
\usepackage{algorithmic}
\usepackage{hyperref}
\ifpdf
  \DeclareGraphicsExtensions{.eps,.pdf,.png,.jpg}
\else
  \DeclareGraphicsExtensions{.eps}
\fi

\usepackage{enumitem}
\setlist[enumerate]{leftmargin=.5in}
\setlist[itemize]{leftmargin=.5in}

\newtheorem{definition}{Definition}
\newtheorem{theorem}{Theorem}
\newtheorem{remark}{Remark}
\newtheorem{lemma}{Lemma}
\newtheorem{corollary}{Corollary}

\usepackage{amsopn}

\begin{document}
\begin{center}
{\Large \bf Oscillatory collision approach in the Earth-Moon restricted three body problem}

 \vskip 0.5cm
{\large Maciej J. Capi\'nski \footnote{
    MJC has been partially supported by the the Polish National Science Center (NCN) grants 2019/35/B/ST1/00655 and 2021/41/B/ST1/00407.}, Aleksander Pasiut \footnote{
    AP has been partially supported by the NCN grant 2021/41/B/ST1/00407.}}

 \vskip 0.2cm
 { Faculty of Applied Mathematics,\\
 AGH University of Kraków}\\
{\small  Mickiewicza 30, 30-059 Krak\'ow, Poland} \\
e-mail: \texttt{maciej.capinski@agh.edu.pl, pasiut@agh.edu.pl}

\vskip 0.5cm

 \today
\vskip 0.5cm

\end{center}

\begin{abstract}
    We consider the Earth-Moon planar circular restricted three body problem and present a proof of the existence orbits, which approach arbitrarily close to one of the primary masses, and at the same time after each approach they move away from the mass to a prescribed distance. In other words the orbits oscillate between being arbitrarily close to collision and away from it. 
    We achieve our goal with the use of topological tools combined with rigorous interval computations. We use the Levi-Civita regularization and validate that the dynamics in the regularized coordinates leads to a good topological alignment between various sets. We then perform shadowing arguments that this leads to the required dynamics in the original coordinates of the system.
\end{abstract}

\vspace{0.3cm} \noindent {\bf Keywords:}  celestial mechanics, collisions, computer assisted proofs

\vspace{0.3cm} \noindent {\bf AMS classification:} 37C29, 37J46, 70F07

\vskip\baselineskip

\section{Introduction}
Despite being formulated in the 17th century, the planar circular restricted three body problem (PCR3BP) still remains a source of interesting dynamical systems problems. The use of modern numerical tools, including computer assisted interval arithmetic machinery, makes it possible to obtain new results for this classical problem and improve the description and understanding of its properties.

The PCR3BP describes the motion of a massless particle under the gravitational pull of two large masses, called primaries, which move on circular orbits on the same plane as the massless particle.
In 1922, J. Chazy \cite{MR1509241} classified the possible final motions that a trajectory $q(t)$ of a massless body in the PCR3BP may have when the time $t \to \pm \infty$ as:
\begin{itemize}
    \item $H^{\pm}$ (hyperbolic): 
    $ \|q(t)\|\to\infty $ and $\| \dot q(t)\| \to c>0 $ as $t\to \pm \infty, $
    \item $P^{\pm}$ (parabolic):
    $ \|q(t)\|\to\infty $ and $\| \dot q(t)\| \to 0 $ as $ t\to \pm \infty,$
    \item $Os^{\pm}$ (oscillatory): 
    \[
        \limsup_{t \to \pm \infty} \| q(t) \|= + \infty\text{ and }
        \liminf_{t \to \pm \infty} \| q(t) \|<+\infty,
    \]
    \item $B^{\pm}$ (bounded):
    $ \limsup_{t\to \pm \infty} \|q(t)\| < \infty. $
\end{itemize}
It is a natural extension to include also a possibility of collision in this classification. Denoting the distance from $q(t)$ to its closest primary as $r(t)$ we can further define:
\begin{itemize}
    \item $C^{\pm}$ (colliding): 
    $r(t)\to 0 $ as $t\to t_c^{\pm}$, for some collision time $t_c,$
    \item $Oc^{\pm}$ (oscillating to collision):
	\[
        \liminf_{t \to \pm \infty} r(t) = 0\text{ and }\limsup_{t \to \pm \infty} r(t) >0,
    \]
    \item $A^{\pm}$ (away from collision):
    $\liminf_{t \to \pm \infty} r(t) > 0$.
\end{itemize}
This research paper is a part of a project to develop a methodology for proving that all of these types of motions co-exist and can be combined together. The emphasis for us is for the method to be applicable in {\em non-perturbative} regimes and to be applicable to energy levels and parameter domains relevant to celestial mechanics. In this paper we focus on the motions associated with collisions and our main result is:
\begin{theorem}
    \label{thm:primary_result}
    Consider the PCR3BP with $\mu=1/82$, which is an approximation of the Earth-Moon system. Then,
    \begin{equation}
        \label{eq:main}
        X^- \cap Y^+ \ne \emptyset,
    \end{equation}
    for $X,Y \in \{C,Oc,A\}$. 

    In particular, we also prove that for every $\varepsilon>0$ there exists a periodic orbit satisfying
    \[
        \min_{t\in \mathbb{R}} r(t)<\varepsilon \qquad \mbox{and} \qquad \max_{t\in \mathbb{R}} r(t) > 1.
    \]
\end{theorem}

In a related paper \cite{MR4391693} it has been proven that (\ref{eq:main}) holds for $X,Y\in \{H,P,Os,B\}$. The next step, which is part of ongoing work, will be to combine the two results. Our methods are developed towards such end.

Our construction for the proof of Theorem \ref{thm:primary_result} is based on the family of Lyapunov orbits which originate around the libration fixed point located between the Earth and the Moon. Such family of periodic orbits is parameterized by the energy, with each orbit lying on a different energy level. The periodic orbits extend far from the fixed point, until they collide with Earth, forming an ejection/collision orbit \cite{Broucke, MR4576879}. A surprising feature is that the family of Lyapunov orbits does not terminate at such collision, it extends beyond, but each periodic orbit makes an additional loop around Earth. (See Figure \ref{fig:lyapunov_orbits}.) 

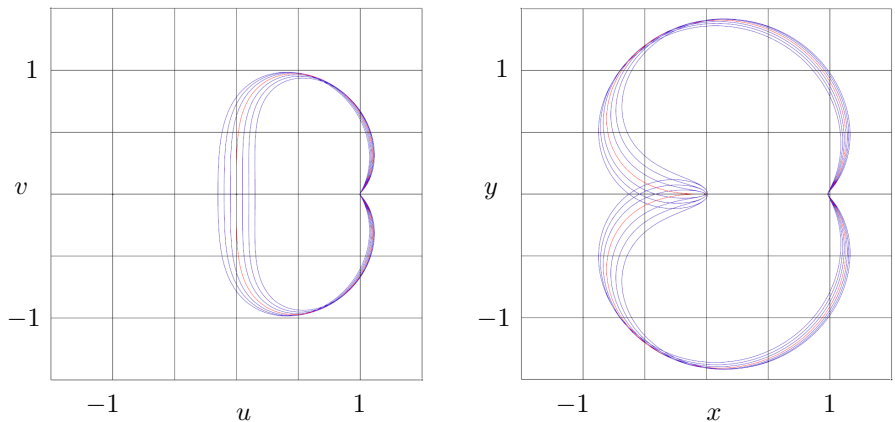
\begin{figure}[h]
    \centering
    \begin{overpic}[width=5cm]{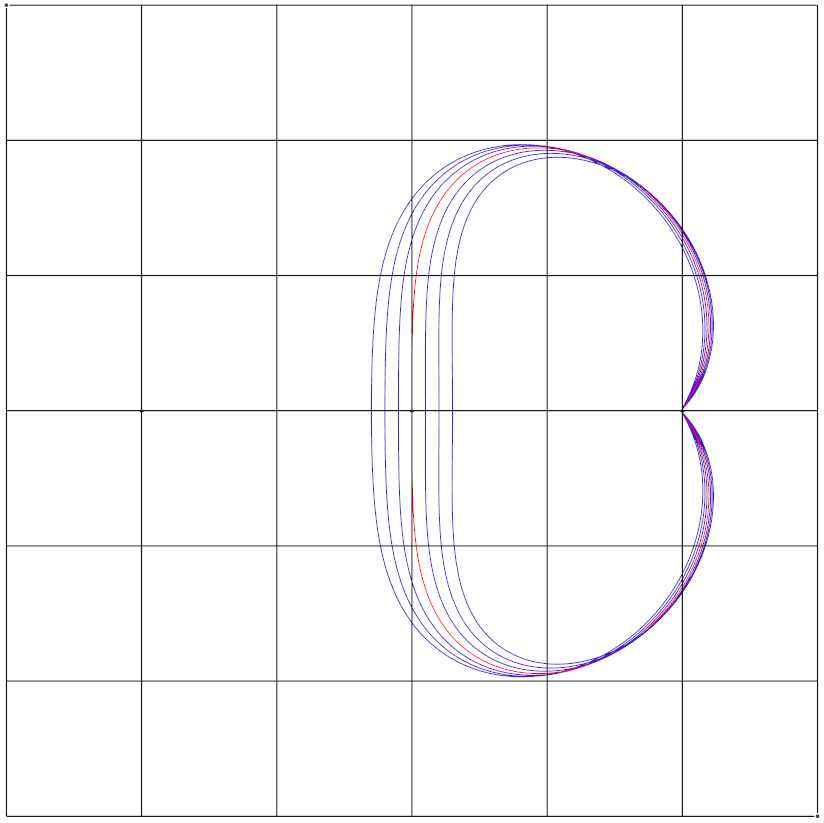}    
        \put (10,-8) {$-1$}
        \put (50,-10) {$u$}
        \put (81,-8) {$1$}
        \put (-6,81) {$1$}
        \put (-9,50) {$v$}
        \put (-11,15) {$-1$}
    \end{overpic}
    \begin{tikzpicture}
        \draw[draw=white] (0,0) rectangle (1,0);
    \end{tikzpicture}
    \begin{overpic}[width=5cm]{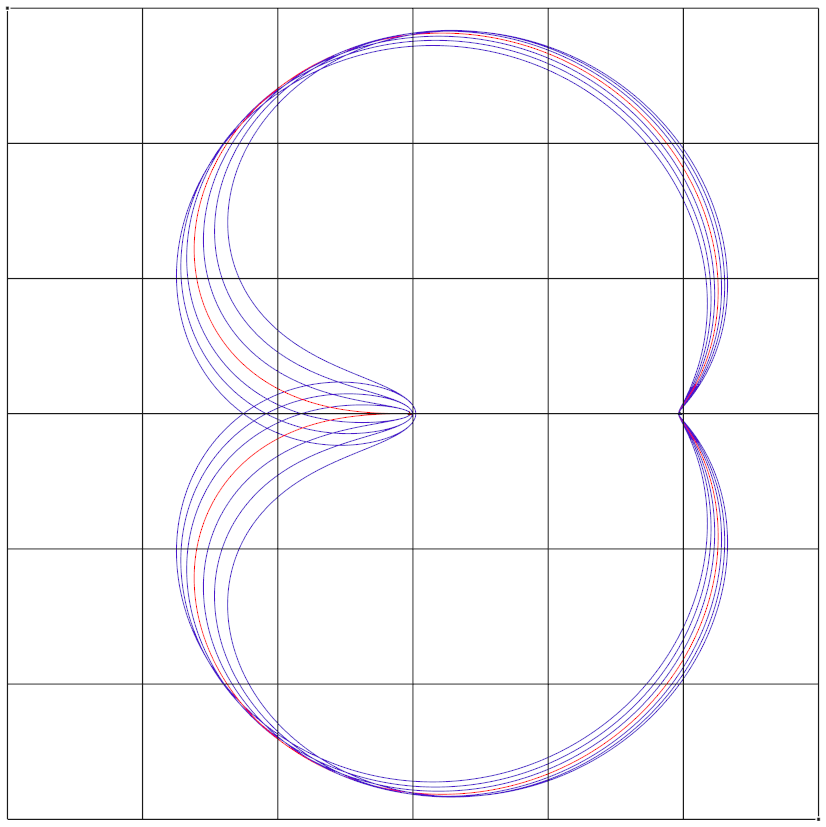}
        \put (10,-8) {$-1$}
        \put (50,-10) {$x$}
        \put (81,-8) {$1$}
        \put (-6,81) {$1$}
        \put (-9,50) {$y$}
        \put (-11,15) {$-1$}
    \end{overpic}
    \vspace{8pt}
    \caption{Lyapunov orbits in the regularized coordinates, where the collision is at the origin (left) and in the original coordinates, where the collision is $(x_2,0)=(-\mu_1,0)$ at the tip of the ``left wedge" (right). The ejection/collision orbit is depicted in red. The orbits pass close to the Moon, which is at $(1,0)$ on the left plot and at $(x_1,0)=(\mu_2,0)$ on the right plot.}
    \label{fig:lyapunov_orbits}
\end{figure}

In our construction we exploit the fact that each Lyapunov orbit, when considered on the constant energy manifold, is hyperbolic. The stable and unstable manifolds of a given Lyapunov orbit can intersect, leading to a topological horseshoe and chaos. This is also the case for the ejection/collision orbit. It is difficult though to talk about a stable and unstable manifold of an ejection/collision orbit in the original coordinates, but it makes perfect sense to do so in the regularized Levi Civita coordinates, since in these coordinates the ejection/collision orbit is a well behaved hyperbolic periodic orbit. We use the regularized coordinates to position sets for the construction of a horseshoe, which we place on suitable Poincar\'e sections. We then use the method of covering relations introduced by Gidea and Zgliczy\'nski \cite{ZGLICZYNSKI200432} to show that we can shadow chaotic orbits passing through such sets. A slightly more delicate issue is to ensure that such shadowing orbits do not collide with Earth. We therefore provide an analytic argument ensuring that suitable sequences of covering relations can be constructed arbitrarily close to the collision, and that the sets used for the construction are separated from the collision.

The idea to use stable/unstable manifolds of invariant objects to prove oscillatory motions dates back to the works of Sitnikov \cite{MR0127389}, who constructed oscillatory motions to infinity in the 1960’s, for a symmetric restricted spatial 3-body problem we nowadays refer to as the Sitnikov model. The use of the intersections of stable/unstable manifolds to infinity was introduced by Moser \cite{Moser01} and has also been used in a number of works to obtain oscillatory motions to infinity in a number of contexts, including \cite{MR4391693,GalanteK11,GuardiaMS16,MR4265664,GuardiaSMS17,SimoL80,LlibreS80,Seara20}. Most of these, with to the best our knowledge the exception of \cite{MR4391693,GalanteK11}, use perturbative methods. As a result the orbits obtained are either confined to `small' regions of the phase space or only exist for certain special parameter ranges.

Our paper is related to a number of previous works concerning collision and near collision orbits. What we believe makes our approach different is the fact that we do not use perturbative arguments and establish the result by purely constructive arguments. This allows us to treat the system far from a perturbative regime, for the explicit mass parameter of the Earth-Moon system. 

Our result is related to the work of Font, Nunes, and Simó \cite{MR1877971, MR2475705}. They study the existence of chaotic invariant sets, which contain orbits that make infinitely many near collision approaches to the smaller body in the PCR3BP. In their work the mass ratio between the primaries is treated as the perturbation parameter. The method works by performing perturbative expansions of Poincar\'e maps, and as in our work, showing the existence a horseshoe involving near collision approaches. The method is constructive, and  provides numerical evidence that their expansions are applicable for the mass parameter up to, $10^{-3}$. This is a concrete and physically meaningful value, but the evidence is numerical, and the value is below the one we use, which is the Earth-Moon system.

Similar results to our Theorem \ref{thm:primary_result} have also been stablished using KAM arguments combined with Levi-Civita coordinates. For instance, in \cite{MR0967629} Chenciner and Llibre prove that in the regularized coordinates the collision circle intersects transversally with invariant tori; the authors refer these as ``punctured invariant tori". (The tori are punctured by the collision set.) On the invariant tori the dynamics is conjugate to irrational rotation, which combined with the collision circle leads to the existence of orbits which pass arbitrarily close to collision. The arguments work for any mass ration but the Jacobi constant needs to be sufficiently large. Punctured tori  are also considered by Féjoz \cite{MR1849229, MR1919782} in the setting of the averaged four body problem, and the result is obtained for a parameter regime in which the problem can be considered as a perturbation of two uncoupled Kepler problems. In a related work by Zhao \cite{MR3417880} it is shown that there exists a positive measure of the punctured tori in the spacial restricted three body problem.

Our result is also related to the work by Bolotin and Mackay \cite{MR1805879} for the PCR3BP. They use variational methods to prove the existence normally hyperbolic invariant manifolds, whose stable/unstable manifolds intersect transversally near an ejection/collision orbit. This also leads to the orbits which oscillate arbitrarily close to the collision. The result holds for an explicit interval of energies, but the mass ratio is treated as a small perturbation parameter. The same authors extend this approach to the spatial problem, \cite{MR2245344} and to the elliptic restricted three body problem \cite{MR2331205}.

Regarding our program to develop a method for combining all types of motions, it is important to point to the work of Moeckel  \cite{MR2350333} who has proved the existence of oscillatory motions and close to collision orbits via symbolic dynamics for the three body problem relying on triple collision approaches. His construction is based on the method of correctly aligned windows. (Covering relations we use in this paper are synonymous with correctly aligned windows.) Due to the construction relying on a triple collision the result is applicable for sufficiently small total angular momentum.

Our paper is organized as follows. In section \ref{sec:prel} we give preliminaries, which cover the PCR3BP, Levi-Civita regularization, interval Newton method and covering relations. Section \ref{sec:Poincare-maps} contains a description of how we set up Poincar\'e sections and covering relations on a constant energy level. In section \ref{sec:dynamics_in_regularized_system} we construct covering relations for a regularized system in Levi-Civita coordinates. In section \ref{sec:collision-approach} we show that the symbolic dynamics stemming from the covering relations from section \ref{sec:dynamics_in_regularized_system} leads to the motions from Theorem \ref{thm:primary_result}.

\section{Preliminaries\label{sec:prel}}

\subsection{The planar circular restricted 3-body problem\label{subsec:PCR3BP}}

Planar circular restricted 3-body problem (PCR3BP) is a celestial mechanics model in which three point masses are placed on a two-dimensional plane and they are interacting with each other due to the gravitational force. The first two masses have values $\mu_{1}$ and $\mu_{2}$, while the third one is infinitesimal. The third mass is referred to as the “test particle”, while the first two large masses are called “primaries”. We assume that the primaries are moving along circular orbits around the center of mass of the system with angular velocity equal to one. The point of interest is the movement of the test particle.

Let us introduce a coordinate system whose origin coincides with the center of mass of the primaries and which is co-rotating with them in such a way that the masses become stationary at points $(x_{1},0)$ and $(x_{2},0)$, respectively. We assume that units of mass, length and time are such that $x_{1}=\mu_{2}$, $x_{2}=-\mu_{1}$ and $\mu_{1}+\mu_{2}=1$.

The movement of the test particle is driven by the Hamiltonian $H$, which is defined as:
\begin{equation}
    H(x,y,p_{x},p_{y})=\frac{1}{2}(p_{x}^{2}+p_{y}^{2})+yp_{x}-xp_{y}-\sum_{i\in\{1,2\}}\mu_{i}\big((x-x_{i})^{2}+y^{2}\big)^{-\frac{1}{2}}.
\end{equation}
The $x$,$y$ are the position coordinates of the test particle and $p_{x}$,$p_{y}$ are their associated momenta.
\begin{remark}
    The phase space in this setup is $\mathbb{R}^{4}$ but with exclusion of two 2-dimensional planes: 
    \[
        \big\{(x_{i},0,p_{x},p_{y}):p_{x},p_{y}\in\mathbb{R},i\in\{1,2\}\big\},
    \]
    for which the Hamiltonian $H$ is not defined and which are interpreted as collisions of the test particle with the respective primaries.
\end{remark}

The Hamiltonian $H$ defines the following system of ordinary differential equations, which together with an initial position and momentum of the test particle uniquely determine its evolution:
\begin{equation}
    \label{eq:std_flow}
    \frac{d}{dt}(x,y,p_{x},p_{y})=J\nabla H(x,y,p_{x},p_{y}),
\end{equation}
where
\begin{equation}
    \label{eq:J_matrix_def}
    J=\begin{bmatrix}0 & I_{2}\\-I_{2} & 0 \end{bmatrix},\quad
    I_{2}=\begin{bmatrix}1 & 0\\0 & 1\end{bmatrix}.
\end{equation}
In the expanded form, the equation (\ref{eq:std_flow}) reads:
\begin{align*}
    \frac{dx}{dt} & =p_{x}+y,\\
    \frac{dy}{dt} & =p_{y}-x,\\
    \frac{dp_{x}}{dt} & =p_{y}-\sum_{i\in\{1,2\}}\mu_{i}(x-x_{i})\big((x-x_{i})^{2}+y^{2}\big)^{-\frac{3}{2}},\\
    \frac{dp_{y}}{dt} & =-p_{x}-\sum_{i\in\{1,2\}}\mu_{i}y\big((x-x_{i})^{2}+y^{2}\big)^{-\frac{3}{2}},
\end{align*}
where $t$ is time. For $q\in\mathbb{R}^{4}$ we shall write $\Phi_{t}(q)$ for the flow induced by (\ref{eq:std_flow}). 

Let us notice that the differential equation (\ref{eq:std_flow}) features a time-reversal symmetry with respect to the function $S$ which is defined by:
\begin{equation}
    \label{eq:S_symmetry}
    S(q_{1},q_{2},q_{3},q_{4})=(q_{1},-q_{2},-q_{3},q_{4}).
\end{equation}
The time-reversal symmetry is
\begin{equation}
    \label{eq:S_symmetry-flow}
    S\big(\Phi_{t}(q)\big)=\Phi_{-t}\big(S(q)\big).
\end{equation}

\subsection{Levi-Civita regularization}

In section (\ref{subsec:PCR3BP}) it was stated that the phase space of the standard PCR3BP is $\mathbb{R}^{4}$ with exclusion of two 2-dimensional planes that correspond to collisions of test particle with respective primaries. Such situation is problematic because it makes it impossible to apply numerical tools to analyze trajectories that pass in the neighborhood of the collisions, or which reach the collisions.

A solution to this problem is the Levi-Civita regularization \cite{MR1555161, 10.1007/BF02418577}. In summary, the method is based on introducing a suitable coordinate system, together with a suitable change of time, after which we obtain system driven by a Hamiltonian $\Gamma$. The two systems are equivalent, up to the coordinate and time change.

Let $x_{i}$ for one of the $i\in\{1,2\}$ and $h\in\mathbb{R}$ be fixed. Let us introduce a coordinate change $(x,y,p_x,p_y)=\gamma_i (u,v,p_u,p_v)$ such that:
\begin{equation}
    \label{eq:lc_coord_change}
    \begin{aligned}
        x & =u^{2}-v^{2}+x_{i},\\
        y & =2uv,\\
        p_{x} & =\frac{1}{2}\frac{up_{u}-vp_{v}}{u^{2}+v^{2}},\\
        p_{y} & =\frac{1}{2}\frac{vp_{u}+up_{v}}{u^{2}+v^{2}}.
    \end{aligned}
\end{equation}
This coordinate change is defined for points from the set
\[
    \big\{(u,v,p_{u},p_{v})\in\mathbb{R}^{4}:(u,v)\neq(0,0)\big\}.
\]
The image of $\gamma_{i}$ is $\big\{(x,y,p_{x},p_{y})\in\mathbb{R}^{4}:(x,y)\neq(x_{i},0)\big\}$. This coordinate change is designed to remove the singularity from the equation at the point $x_{i}$ . We shall therefore say that it ``regularizes" the system at $x_{i}$ . The collision with $x_{i}$ corresponds to $u=v=0$. At this point the coordinate change (\ref{eq:lc_coord_change}) is not defined, but we shall see that after a suitable change of time the resulting ODE in the regularized coordinates is well defined.
\begin{remark}
    Position change in (\ref{eq:lc_coord_change}) satisfies $x+\mathrm{i}y=(u+\mathrm{i}v)^{2}+x_{i}$. The transformation of momenta is adjusted in such a way, that the coordinate change $\gamma_{i}$ is canonical.
\end{remark}
The regularization also involves a change of time from $t$ to $s$, which is defined by
\begin{equation}
    \frac{dt}{ds}(s)=4\big(u^{2}(s)+v^{2}(s)\big).\label{eq:reg_time_def}
\end{equation}

Let us now define a Hamiltonian $\Gamma_{i,h}$ as:
\begin{equation}
    \Gamma_{i,h}(u,v,p_{u},p_{v}) =\frac{1}{2}(p_{u}^{2}+p_{v}^{2})-2x_{i}(vp_{u}+up_{v})-4\mu_{i}+2(u^{2}+v^{2})\hat{\Gamma}_{i,h}\label{eq:Gamma_i}
\end{equation}
where
\[
    \hat{\Gamma}_{i,h}=vp_{u}-up_{v}-2h-\frac{2\mu_{3-i}}{\sqrt{(u^{2}-v^{2}+\epsilon_{i})^{2}+(2uv)^{2}}},
\]
and
\[
    \epsilon_{i}=\begin{cases}
        +1, & i=1,\\
        -1, & i=2.
    \end{cases}
\]

We consider the Hamiltonian system driven by $\Gamma_{i,h}$ under the time $s$. When written in full form, the ordinary differential equations induced by this system are:
\begin{align*}
    \frac{du}{ds} & =p_{u}+2v(u^{2}+v^{2}-x_{i}),\\
    \frac{dv}{ds} & =p_{v}-2u(u^{2}+v^{2}+x_{i}),\\
    \frac{dp_{u}}{ds} & =4u(2h-vp_{u})+2p_{v}(x_{i}+3u^{2}+v^{2})+\frac{8u\mu_{3-i}\big(1+\epsilon_{i}(u^{2}-3v^{2})\big)}{\big((u^{2}-v^{2}+\epsilon_{i})^{2}+4u^{2}v^{2}\big)^{\frac{3}{2}}},\\
    \frac{dp_{v}}{ds} & =4v(2h+up_{v})+2p_{u}(x_{i}-3v^{2}-u^{2})+\frac{8v\mu_{3-i}\big(1-\epsilon_{i}(v^{2}-3u^{2})\big)}{\big((u^{2}-v^{2}+\epsilon_{i})^{2}+4u^{2}v^{2}\big)^{\frac{3}{2}}}.
\end{align*}

For $w=(u,v,p_{u},p_{v})$ and time $s$ we shall use the notation $\Phi_{s}^{i,h}(w)$ for the flow induced by $\Gamma_{i,h}$.
\begin{remark}
    The phase space of the Hamiltonian $\Gamma_{i,h}$ is $\mathbb{R}^{4}$, but with exclusion of two 2-dimensional planes for which the Hamiltonian $\Gamma_{i,h}$ is not defined. (These planes represent the collision with the second body, which was not regularized.) For $i=1$ we exclude the planes
    \[
        \big\{(0,v,p_{u},p_{v}):v=\pm1,p_{u},p_{v}\in\mathbb{R}\big\}=\gamma_{1}^{-1}\Big(\big\{(x_{2},0,p_{x},p_{y}),p_{x},p_{y}\in\mathbb{R}\big\}\Big),
    \]
    and for $i=2$ we exclude
    \[
        \big\{(u,0,p_{u},p_{v}):u=\pm1,p_{u},p_{v}\in\mathbb{R}\big\}=\gamma_{2}^{-1}\Big(\big\{(x_{1},0,p_{x},p_{y}),p_{x},p_{y}\in\mathbb{R}\big\}\Big).
    \]
 
\end{remark}

The following theorem is the key result in the context of the regularization
method.
\begin{theorem}
    \label{thm:pcr3bp_reg}
    \cite{10.1007/BF02418577} If $w \in \mathbb{R}^4$ is such that $\Gamma_{i,h}(w)=0$ and for every time $\sigma\in[0,s(t)]$ we have $\pi_{u,v}\big(\Phi_{\sigma}^{i,h}(w)\big)\neq(0,0)$, where $\pi_{u,v}$ denotes the projection on coordinates $u,v$, then
    \begin{equation}
        \Phi_{t}\big(\gamma_{i}(w)\big)=\gamma_{i}\big(\Phi_{s(t)}^{i,h}(w)\big).\label{eq:pcr3bp_reg}
    \end{equation}
\end{theorem}
Let us now make some comments regarding this result.

Every point $q$ in the original coordinates corresponds to two points $w_{1}$, $w_{2}$ in regularized coordinates, meaning that $\gamma_{i}(w_{1})=\gamma_{i}(w_{2})=q$. Theorem \ref{thm:pcr3bp_reg} is valid regardles of the choice of $w_{1}$ or $w_{2}$. The key assumption of Theorem \ref{thm:pcr3bp_reg} that $\Gamma_{i,h} (w) = 0$ corresponds to the fact that (\ref{eq:pcr3bp_reg}) holds only for $w$ for which $H\big(\gamma_{i}(w)\big)=h$. We could propagate the flow $\Phi_{i,h}(w)$ for some point $\Gamma_{i,h}(w) \neq 0$, but such trajectory would not correspond to a true physical trajectory of the original system. We can view $\Gamma_{i,h}$ as a family of systems, parameterized by $h$, each providing true trajectories of the original system for the energy level $H=h$.

For S defined in (\ref{eq:S_symmetry}) we have the following time-reversing symmetry
\[
    S\left(\Phi_{s}^{i,h}(w)\right)=\Phi_{-s}^{i,h}\left(S(w)\right).
\]

We finish by noting that the collision with $x_{i}$ in the regularized coordinates corresponds to $u=v=0$. Thus from the condition $\Gamma_{i,h} (w) = 0$, the physical trajectories arriving or ejecting from collision need to satisfy
\begin{equation}
    p_{u}^{2}+p_{v}^{2}=8\mu_{i}. \label{eq:collision-circle}
\end{equation}
This means that in the regularized coordinates a collision is a circle.

\subsection{Interval Newton method}

In this section we describe a useful tool from interval analysis. Let us first set up the needed notation.

We shall refer to a product of closed intervals in $\mathbb{R}^{n}$ as an interval vector. We define an interval enclosure of a set $X\subset\mathbb{R}^{n}$ by:
\[
    [X]=Y_{1}\times...\times Y_{n}\text{ such that }\pi_{i}(X)\subset Y_{i}\text{ for }i\in\{1,...,n\}
\]
where $Y_{1},...,Y_{n}$ are closed real intervals and $\pi_{i}$ is the projection onto the $i$-th component. Note that an interval enclosure is not unique. (We prefer tighter enclosures, since they provide better estimates.) For a set $X\subset\mathbb{R}^{n}$ and for a map $F=(F_{1},...,F_{n}):\mathbb{R}^{n}\to\mathbb{R}^{n}$ we denote an interval enclosure of $F(X)$ as
\[
    \big[F(X)\big]=Y_{1}\times...\times Y_{n}\text{ such that }F_{i}(X)\subset Y_{i}\text{ for }i\in\{1,...,n\}.
\]

A set $A\subset\mathbb{R}^{n\times n}$ is an interval matrix if each of its entries is a closed interval. We say that an interval matrix $A$ is invertible if for every $B\in A$ the matrix $B$ is invertible. We will call an interval matrix $[A^{-1}]$ an interval enclosure of the inverse of $A$ if for every $B\in A$ we have $B^{-1}\in[A^{-1}]$. Note that as in case of interval enclosures in general, the interval enclosure of the inverse of $A$ is not unique, since after it is
enlarged it remains an enclosure of the inverse.

If the map $F$ is $C^{1}$ then we will write $[DF(X)]\subset\mathbb{R}^{n\times n}$ for an interval matrix that satisfies 
\[
    DF(X)\subset[DF(X)].
\]

\begin{theorem}
    \cite{ALEFELD2000421} Let $X$ be an interval vector with nonempty interior, let $x\in\mathrm{Int(X)}$ be some point and let $F:\mathbb{R}^{n}\to\mathbb{R}^{n}$ be $C^{1}$. If for 
    \[
        N(x,X):=x-\big[DF(X)^{-1}\big]\big[F(x)\big]
    \]
    we have $N(x,X)\subset X$, then there exists a point $x^{*}\in N(x,X)$ such that 
    \[
        F(x^{*})=0.
    \]
\end{theorem}
The result can be extended to include a parameter as follows.
\begin{theorem}
    \label{thm:extended_interval_Newton_theorem}
    \cite{MR2652784} Let $X$ be an interval vector with nonempty interior, let $x\in\mathrm{Int(X)}$ be some point, let $I\subset\mathbb{R}$ be an interval and let $F:I\times\mathbb{R}^{n}\to\mathbb{R}^{n}$ be $C^{1}.$ If for 
    \[
        N(x,I,X):=x-\Big[D_{x}F(I,X)^{-1}\Big]\big[F(I,x)\big]
    \]
    we have $N(x,I,X)\subset X$, then there exists a $C^{1}$ function $x^{*}:I\to N(x,I,X)$ such that
    \[
        F\big(\lambda,x^{*}(\lambda)\big)=0\quad\text{for every}\quad\lambda\in I.
    \]
\end{theorem}

\begin{remark}
    Theorem \ref{thm:extended_interval_Newton_theorem} remains valid if $I$ is not an interval, but an interval vector.
\end{remark}
\subsection{Covering relations}
In this section we recall the results form \cite{ZGLICZYNSKI200432}, which provide the primary topological tools for obtaining our results. In this paper it will be sufficient for us to consider the two dimensional case where both unstable and stable directions of h-sets have dimensions one, so the definitions provided below are the special cases of the respective definitions from \cite{ZGLICZYNSKI200432}.

Let $B = [-1,1]\subset \mathbb{R}$ and $\partial B = \{ -1, 1 \}$.

\begin{definition}
    \label{def:h_set}
    An h-set is a pair $N = \left( |N|, c_N \right)$ where $|N|$ is a compact subset of $\mathbb{R}^2$ and $c_N : \mathbb{R}^2  \to  \mathbb{R}^2$ is a homeomorphism such that:
    \[
        c_N \left( B \times B \right) = |N|.
    \]
    The subset $|N|$ is called a support of an h-set.
\end{definition}

We denote (see Figure \ref{fig:covering}):
\[
    N_c = B \times B, \quad
    N_c^- = \partial B \times B, \quad 
    N_c^+ = B \times \partial B,
\]
\[
    N^- = c_N ( N_c^- ), \quad
    N^+ = c_N ( N_c^+ ),
\]
\[
    N^{l} = c_N \big( \{ -1 \} \times B \big), \quad
    N^{r} = c_N \big( \{ +1 \} \times B \big).
\]
We shall refer to $N^-$ as the exit set.

In this paper the supports of all the the h-sets we consider will be quadrangles. We will frequently refer to the support as the h-set and neglect to write out the homeomorphism $c_N$, which simply maps the square $N_c$ into a quadrangle, whenever it will be clear from the context which edges of the quadrangle represent the exit set.

\begin{figure}
\begin{center}
	\begin{overpic}[width=13cm]{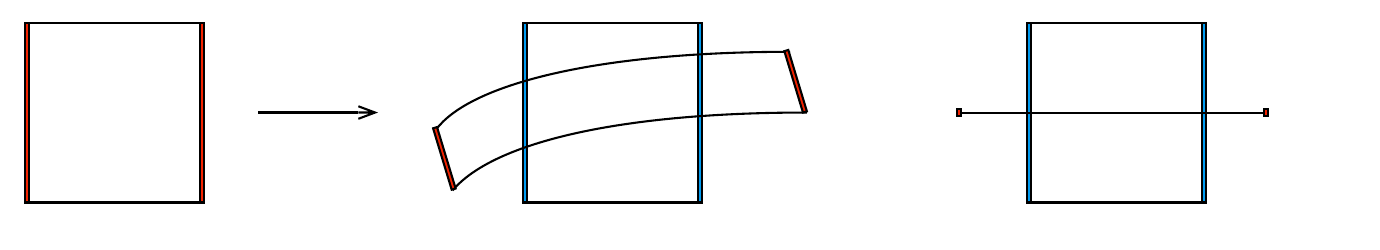}
	\put (7,0){$N_c$}
	\put (43,0){$M_c$}
	\put (79,0){$M_c$}
	\put (22,12){$f_c$}
	\put (55,16){$f_c(N_c)$}
	\put (90,12){$h(1,N_c)$}
	\end{overpic}
\end{center}
\caption{A depiction of a covering relation $N \xRightarrow{f} M$. The exit sets $N_c^-$ and $M_c^-$ are in red and blue, respectively. On the right we have the final result of the homotopy from Definition \ref{def:covering-relation}.\label{fig:covering}}
\end{figure}
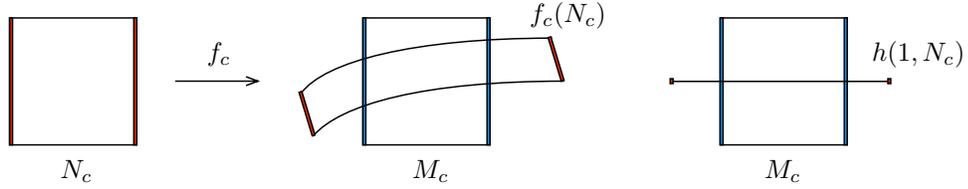

\begin{definition}
    \label{def:covering-relation}
    Assume $N$, $M$ are h-sets. Let $f : N \to \mathbb{R}^n$ be a continuous map, such that the map $f_c = c_M^{-1} \circ f \circ c_N : N_c \to \mathbb{R}^2$ is well-defined and continuous. We say that
    \[
        N \xRightarrow{f} M
    \]
    ($N$ $f$-covers $M$) iff the following conditions are satisfied (fee Figure \ref{fig:covering}):
    
    \begin{enumerate}
        \item There exists a continuous homotopy $h : [0;1] \times N_c \to \mathbb{R}^2$, such that the following conditions hold true:
        \begin{align}
            \begin{aligned}
                h_0 &= f_c, \\
                h \big( [0;1], N_c^- \big) \cap M_c &= \emptyset, \\
                h \big( [0;1], N_c \big) \cap M_c^+ &= \emptyset. \\
            \end{aligned}
        \end{align}
        \item There exists a real number $a > 1$, such that:
        \[
            h(1, (p, q)) = (ap, 0), \quad p,q \in B.
        \]
    \end{enumerate}
\end{definition}

In our computer assisted proofs, we assert the existence of the covering relations by checking the following conditions:
\begin{align}
    \label{eq:covering_relation_conditions}
    \begin{aligned}
        \pi_{2}\big(f_c(N_c)\big) & \subset B,\\
        \pi_{1}\big(f_c(N_c^{l})\big) & < -1,\\
        \pi_{1}\big(f_c(N_c^{r})\big) & > +1,
    \end{aligned}
\end{align}
which we refer to as the contraction condition, the left expansion condition and the right expansion condition respectively.

\begin{definition}\label{def:NT}
    Let $N$ be an h-set. We define the h-set $N^T$ such that the support of $N^T$ is equal to the support of $N$ and  the homeomorphism $c_{N^T}$ is such that:
    \[
        c_{N^T} = c_N \circ j,
    \]
    where $j : \mathbb{R}^2 \to \mathbb{R}^2$ is given by $j(p,q) = (q,p)$.
\end{definition}

\begin{definition}\label{def:back-covering-relation}
    Assume $N$, $M$ are h-sets. Let $g$ be a function such that $g^{-1} : M \to \mathbb{R}^2$ is well defined and continuous. We say that
    \[
        N \xLeftarrow{g} M
    \]
    ($N$ $g$-backcovers $M$) iff $M^T \xRightarrow{g^{-1}} N^T$.
\end{definition}

The following theorem will be our main tool for constructing our oscillating motions.
\begin{theorem}
    \label{thm:generic_covering_relations}
    \cite{ZGLICZYNSKI200432} Let $k \in \mathbb{N}$, $N_0, ..., N_{k-1}$ be a sequence of h-sets, let $N_k=N_0,$ and $f_{1}, \dots, f_{k}$ be a sequence of functions such that
    \[
    	N_{i-1} \xLeftarrow{f_i} N_{i} \qquad \mbox{or} \qquad N_{i-1} \xRightarrow{f_i} N_{i} \qquad \mbox{for }i\in\{1,\ldots,k\}.
    \]
    Then, there exists a point $x$ in the interior of $N_0$ such that:
    \[
        f_{i}\circ f_{i-1}\circ...\circ f_{1}(x) \in N_i, \qquad \mbox{for }i=1,\ldots,k-1,
    \]
    and
    \[
        f_{k}\circ f_{k-1}\circ...\circ f_{1}(x)=x.
    \]
\end{theorem}

The above theorem can be extended (we will use both versions in our arguments) by introducing the following two notions.
\begin{definition}\label{def:horizontal-disc} Let $N$ be an h-set. Let $b:[-1,1]\to |N|$ be continuous and let $b_c=c_N^{-1}\circ b$. We say that $b$ is a {\em horizontal disc} in $N$ if there is a homotopy $h:[0,1]\times[-1,1]\to N_c$ such that
\begin{align*}
	h_0&=b_c, \\
	h_1(x)& = (x,0), \quad \mbox{for all }x\in [-1,1], \\
	h([0,1],x)&\subset N_c^-, \qquad \mbox{for }x\in\{-1,1\}.
\end{align*}
\end{definition}
\begin{definition}\label{def:vertical-disc} Let $N$ be an h-set. Let $b:[-1,1]\to |N|$ be continuous and let $b_c=c_N^{-1}\circ b$. We say that $b$ is a {\em vertical disc} in $N$ if there is a homotopy $h:[0,1]\times[-1,1]\to N_c$ such that
\begin{align*}
	h_0&=b_c, \\
	h_1(y)& = (0,y), \quad \mbox{for all }y\in [-1,1], \\
	h([0,1],y)&\subset N_c^+, \qquad \mbox{for }y\in\{-1,1\}.
\end{align*}
\end{definition}
With these notions we can formulate the following result:
\begin{theorem}\cite{MR2494688}
    \label{thm:cover-discs}
    Let $k \in \mathbb{N}$, $N_0, ..., N_{k-1}$ be a sequence of h-sets and $f_{1}, \dots, f_{k}$ be a sequence of functions such that
    \[
    	N_{i-1} \xLeftarrow{f_i} N_{i} \qquad \mbox{or} \qquad N_{i-1} \xRightarrow{f_i} N_{i} \qquad \mbox{for }i\in\{1,\ldots,k\}.
    \]
    Assume that $b_h$ is a horizontal disc in $N_0$ and $b_v$ is a vertical disc in $N_k$.
    
    Then, there exists a point $x$ in the interior of $N_0$ such that:
    \begin{align*}
    	x&=b_h(t) &\mbox{for some }t\in (-1,1), \\
	    f_{i}\circ f_{i-1}\circ...\circ f_{1}(x) &\in N_i, &\mbox{for }i=1,\ldots,k-1,\,\,\, \\
	    f_{k}\circ f_{k-1}\circ...\circ f_{1}(x)&=b_v(s) &\mbox{for some }s\in(-1,1).
    \end{align*}
  
\end{theorem}

\section{Poincar\'e maps for a flow restricted to a constant energy level}
\label{sec:Poincare-maps}
In our work we introduce Poincar\'e sections and associated Poincar\'e maps in order to reduce the original continuous time system to a discrete time system. In this section we present the associated notation which we will use later on.

Consider a flow $\phi_{t}(x)$ in\footnote{The method described in this section can be extended in a natural way to $\mathbb{R}^n$ but we stick to $\mathbb{R}^{4}$ since this is the dimension of our problem.} $\mathbb{R}^{4}$ and a section $\Sigma\subset\mathbb{R}^{4}$. (In our application we will work in $\mathbb{R}^{4}$, and we will use $\Phi_{s}^{i,h}$ as the flow. The $\Sigma$ which we consider will be three dimensional affine sections in $\mathbb{R}^{4}$.) We denote by $T_{\Sigma}(w)$ the time needed to approach the section along the flow
\begin{equation}
    \label{eq:flow_time_to_section}
    T_{\Sigma}(w)=\inf\{t>0:\phi_{s}(w)\in\Sigma\}.
\end{equation}
Whenever we use the notation $T_{\Sigma}(w)$ we implicitly assume that $w\in\mathbb{R}^{4}$ lies in a domain where $T_{\Sigma}$ is well defined. We define a map $P_{\Sigma}:\mathbb{R}^{4}\to\Sigma$ as
\[
    P_{\Sigma}(w)=\phi_{T_{\Sigma}(w)}(w).
\]
We will refer to $P_{\Sigma}$ as the ``Poincar\'e map" to the section $\Sigma$.

\subsection{Parameterization of a section restricted to a constant energy level}
\label{sec:Sections-on-energy}
For the discussion from this section consider a Hamiltonian $\Gamma:\mathbb{R}^{4}\to\text{\ensuremath{\mathbb{R}}}$ and a vector field 
\[
    \dot{w} = J\nabla \Gamma(w).
\]
(Here $\Gamma$ can be some abstract Hamiltonian, but we choose the notation to coincide wit $\Gamma_{i,h}$ since we will apply this setup to the flow in regularized coordinates.)

We shall consider the following local coordinates at a chosen point $w\in\mathbb{R}^{4}$. Consider $A\in\mathbb{R}^{4\times4}$ of the form 
\begin{equation}
    \label{eq:A-form}
    A=\left[
        \hat{u} \quad
        \hat{s} \quad
        J\nabla \Gamma(w)  \quad
        \nabla \Gamma(w)  \right],
\end{equation}
where the above notation means that the vectors $\hat{u}$, $\hat{s}$, $J\nabla \Gamma(w)$, $\nabla \Gamma(w)$ are the columns of the matrix $A$. Intuitively, the notation should be understood as follows. The $\hat{u}$ stands for the ``unstable" coordinate at $w$, $\hat{s}$ stands for the ``stable" coordinate at $w$. The vector $\nabla \Gamma(w)$ is associated with the ``direction of the energy" and $J\nabla \Gamma(w)$ is the vector field at $w$.
\begin{remark}
    In our applications the point $w$ will be positioned at (or close to) some hyperbolic periodic orbit. The vectors $\hat{u}$ and $\hat{s}$ are chosen to correspond to the linearized expansion and contraction along the periodic orbit at $w$. (The $\hat u, \hat s$ do not need to match perfectly with the unstable and stable bundles. For our needs it will be sufficient if they are ``roughly" aligned.)
\end{remark}

At the chosen point $w$ we consider a local change of coordinates $\Lambda:\mathbb{R}^{4}\to\mathbb{R}^{4}$ defined as 
\begin{equation}
    \label{eq:Lambda-def}
    \Lambda(x) = w + Ax.
\end{equation}
We define a section at $w$ as 
\begin{equation}
    \label{eq:section-choice}
    \Sigma=\big\{ p\in\mathbb{R}^{4}:\langle J\nabla \Gamma(w),p-w\rangle=0\big\}. 
\end{equation}
We see that $\Sigma$ is a three dimensional affine section at $w$ in $\mathbb{R}^{4}$. Observe that by the definition of $\Sigma$, the flow in the neighborhood of $w$ will be transversal to $\Sigma$.

We will investigate the section $\Sigma$ restricted to $\{\Gamma=0\}$. To do so we implicitly define $E:\mathbb{R}^{2}\to\mathbb{R}^{2}$ to be a function satisfying
\begin{equation}
    \label{eq:E-implicit-def}
    \Gamma(\Lambda(z,E(z)))=0, \qquad\text{and}\qquad \Lambda(z,E(z))\in\Sigma.
\end{equation}

\begin{remark}
    The function $E$ depends on the choice of $w$ and $A$.
\end{remark}
\begin{remark}
    \label{rem:E-from-Interval-Newton}
    Note that if we fix $z\in\mathbb{R}^{2}$ then $E(z)$ is the zero of 
    \[
        \mathbb{R}^{2} \ni e\mapsto\left(\Gamma(\Lambda(z,e)),
        \langle J\nabla \Gamma(w),A(z,e)\rangle\right)\in\mathbb{R}^{2}.
    \]
    This means that for a fixed $w$, fixed $A$ and for fixed $z$ the $E(z)$ can be computed in interval arithmetic by means of the interval Newton method.
\end{remark}
What we have achieved by introducing $E$ is that
\[
    \Lambda(z,E(z))\in\{\Gamma=0\}\cap\Sigma.
\]
The function $\psi:\mathbb{R}^{2}\to\{\Gamma=0\}\cap\Sigma$ defined as
\begin{equation}
    \label{eq:local-parametrization-choice}
    \psi(z):=\Lambda(z,E(z))
\end{equation}
is a parametrization of $\{\Gamma=0\}\cap\Sigma$. We can think of $z$ as our local coordinates on $\{\Gamma=0\}\cap\Sigma$. 

Intuitively, restricting to the section $\Sigma$ reduces the dimension from $4$ to $3$, and restricting to the fixed energy level further reduces the dimension to $2$. The $z$ are these remaining two dimensional coordinates.

\begin{remark}
    \label{rem:local-map-choice}
    The function $\psi$ is uniquely determined by the choice of the triple $(w,\hat u, \hat s)$.
\end{remark}

Below lemma gives us the inverse of $\psi$.
\begin{lemma}
    \label{lem:inverse-psi}
    For $p\in\{\Gamma=0\}\cap\Sigma$ 
    \[
        \psi^{-1}(p)=\pi_{1,2}\Lambda^{-1} (p).
    \]
\end{lemma}
\begin{proof}
    Consider $e$ to be defined as $e=\pi_{3,4}\Lambda^{-1}(p).$ It is enough to show that $e=E(z)$, since if this is the case then we will have that
    \[
        p = \Lambda(z,e) = \Lambda(z,E(z)) = \psi(z).
    \]
    To show that $e=E(z)$ we need to ensure (\ref{eq:E-implicit-def}), namely that
    \[
        \Gamma(\Lambda(z,e))=0,\qquad\text{and}\qquad \Lambda(z,e) \in \Sigma.
    \]
    Both above conditions are clearly satisfied since $\Lambda(z,e) = p \in \{\Gamma=0\}\cap\Sigma$. This concludes our proof.
\end{proof}

The above introduced $\psi$ might not seem the most direct approach for introducing a parametrization of a section on a fixed energy level, but it serves our purposes for the following reasons. The change of coordinates $\psi$ is readily computable by means of the Interval Newton method (see Remark \ref{rem:E-from-Interval-Newton}), and also $\psi^{-1}$ is straightforward to compute by means of Lemma \ref{lem:inverse-psi}. Moreover, in the local coordinates given by our parametrizations the Poincar\'e maps will be well aligned with the dynamics. This is discussed in more depth in the following subsection.

\subsection{Covering relations between sections} 
Let $P : \Sigma_1 \to \Sigma_2$ be a Poincar\'e map where $\Sigma_1$ and $\Sigma_2$ are two sections, with the local coordinates $\psi_1$ and $\psi_2$ constructed on $\Sigma_1\cap \{\Gamma=0\}$ and $\Sigma_2\cap \{\Gamma=0\}$, respectively, as described in the previous subsection. The Poincar\'e (or section-to-section) map expressed in the local coordinates, i.e.
\begin{equation}
    \label{eq:local-poinc-map}
    P_{21}:={(\psi_2)}^{-1} \circ P \circ \psi_1 : \mathbb{R}^2 \to \mathbb{R}^2
\end{equation} 
will be well aligned with the dynamics. What we mean by this is that the image of the point $(0,0)$ under local Poincar\'e map will be close to $(0,0)$, and for $z=(z_1, z_2)$, the first coordinate $z_1$ will be expanding for the local map and $z_2$ will be contracting. Also the derivatives of $P_{21}$ will be close to diagonal. 

\begin{definition}
    \label{def:h_set_on_section}
    Let $\mathcal{N} := (N, \psi)$ be a pair that consists of an h-set $N$ and a parametrization $\psi$ of section $\{\Gamma = 0\} \cap \Sigma$. We shall refer to object $\mathcal{N}$ as an h-set $N$ on section $\Sigma$.
\end{definition}

Let us consider $\mathcal{N}_1 = (N_1, \psi_1)$ and $\mathcal{N}_2 = (N_2, \psi_2)$. We will use the notation:
\begin{equation}
    \label{eq:sec-covering-1}
    \mathcal{N}_1 \xRightarrow{P} \mathcal{N}_2 \qquad\text{and}\qquad
    \mathcal{N}_1 \xLeftarrow{P} \mathcal{N}_2
\end{equation}
to denote the covering relations
\begin{equation}
    \label{eq:sec-covering-2}
    N_1 \xRightarrow{P_{21}} N_2 \qquad\text{or}\qquad
    N_1 \xLeftarrow{P_{21}} N_2 
\end{equation} 
respectively, when it is clear from the context with which Poincar\'e sections and with which local coordinates $P$ is associated.

Consider now the symmetry $S$ of the PCR3BP (\ref{eq:S_symmetry}). Notice, that if $\Gamma\circ S = \Gamma$, which is the case in the PCR3BP, then for a given $\psi:\mathbb{R}^2\to \{\Gamma =0\} \cap \Sigma$, $S\psi$ is a parameterization of $S\Sigma\cap\left\{\Gamma=0\right\}.$ For $\mathcal{N} = (N,\psi)$ we define:
\[
    S\mathcal{N} := (N^T, S\psi).
\]

\begin{lemma}
    \label{lem:S_backsymmetry}
    Assume that the system is $S$-symmetric, meaning that
    \[
        \Gamma\circ S = \Gamma \qquad \mbox{and} \qquad\nabla\Gamma \circ S=-S \circ \nabla\Gamma. 
    \]
    Consider  $\mathcal{N}_1 = (N_1, \psi_1)$ and $\mathcal{N}_2 = (N_2, \psi_2)$, where $N_1$ and $N_2$ are two h-sets on the sections $\Sigma_{1}$ and $\Sigma_{2}$, respectively. The sections are parameterized by $\psi_1$ and $\psi_2$, respectively. Consider also two Poincar\'e maps $P:\Sigma_{1}\rightarrow\Sigma_{2}$ and $P^{S}:S\Sigma_{2}\rightarrow S\Sigma_{1}$. If
    \begin{equation}
        \label{eq:Poinc-cover}
        \mathcal{N}_1 \xRightarrow{P} \mathcal{N}_2
    \end{equation}
    then
    \begin{equation}
        \label{eq:symmetry-back-covering}
        S \mathcal{N}_2 \xLeftarrow{P^S} S \mathcal{N}_1.
    \end{equation}
\end{lemma}

\begin{proof}
    The condition (\ref{eq:symmetry-back-covering}) written in the expanded form reads:
    \begin{equation}
        N_2^T \xLeftarrow{P_{12}^S} N_1^T,
    \end{equation}
    which by the definition of backcovering relation is equivalent to:
    \begin{equation}
        \label{eq:symmetry-cover-req}
        N_1 \xRightarrow{(P_{12}^S)^{-1}} N_2.
    \end{equation}
    From the symmetry of the system we have%
    \[
        S \circ P=(P^{S})^{-1}\circ S,
    \]
    which from (\ref{eq:local-poinc-map}) implies that
    \[
        (P^{S}_{12})^{-1}=\left(  \psi_{2}\right)  ^{-1}\circ S^{-1}\circ (P^{S})^{-1}\circ S\circ\psi_{1}=\left(  \psi_{2}\right)  ^{-1}\circ P\circ\psi_{1}=P_{21},
    \]
    so (\ref{eq:symmetry-cover-req}) follows from (\ref{eq:Poinc-cover}).
\end{proof}

\begin{definition}
    \label{def:self-S-symmetry}
    We will say that $\mathcal{N} = (N,\psi)$ is self $S$-symmetric, if
    \[
        S \circ \psi \circ c_{N^T}=\psi \circ c_N
    \]
    where $N$ is an h-set with a homeomorphism $c_N$ (see Definition \ref{def:NT} for the notation $c_{N^T}$).
\end{definition}

\begin{lemma}
    \label{lem:self-symmetric-covering}
    Assume that $\mathcal{N}$ is a self $S$-symmetric h-set on a Poincar\'e section $\Sigma_N$ and that $\mathcal{M}$ is an h-set (not necessarily self $S$-symmetric) on a Poincar\'e section $\Sigma_M$. Then (here we use the short-hand notation (\ref{eq:sec-covering-1}))
    \begin{align*} 
        \mathcal{N} \xRightarrow{P} \mathcal{M} & \qquad \mbox{implies} \qquad S\mathcal{N} \xRightarrow{P} \mathcal{M}, \\
        \mathcal{M} \xRightarrow{P} \mathcal{N} & \qquad \mbox{implies} \qquad \mathcal{M} \xRightarrow{P} S\mathcal{N}, \\
        \mathcal{N} \xLeftarrow{P} \mathcal{M} & \qquad \mbox{implies} \qquad S\mathcal{N} \xLeftarrow{P} \mathcal{M}, \\
        \mathcal{M} \xLeftarrow{P} \mathcal{N} & \qquad \mbox{implies} \qquad \mathcal{M} \xLeftarrow{P} S\mathcal{N}.
    \end{align*}
\end{lemma}
\begin{proof}
    The result follows directly from the Definitions \ref{def:covering-relation}, \ref{def:back-covering-relation} and \ref{def:self-S-symmetry}. (This is since the maps $f_c$ for the covering relations on the left hand sides are the same as those on the right hand sides.)
\end{proof}

\subsection{Section at collision}\label{sec:section-at-collision} In this section we focus on the Hamiltonian $\Gamma_{2,h}$ from (\ref{eq:Gamma_i}) for the regularised PCR3BP at the second mass.
A special role in our construction will be played by a Poincar\'e section at a collision point chosen as
\begin{equation}
    \label{eq:w0}
	w_{0}:=(0,0,0,\sqrt{8\mu_{2}}). 
\end{equation}
At this point we will consider $\Sigma_0=\{(u,v,p_u,p_v):v=0\}$ and construct $\psi_0:\mathbb{R}^2\to\Sigma_0\cap\{\Gamma_{2,h}=0\}$ so that our choice of $\psi_0$ ensures certain symmetry properties of h-sets placed on such section. The choice of $h$ will be determined later on in our construction.
\begin{remark}
    All the Poincar\'e sections which we will later consider in our construction will be chosen as described in section \ref{sec:Sections-on-energy} except the section $\Sigma_0$ positioned at $w_0$ from (\ref{eq:w0}) with the parametrisation $\psi_0$ as described below.
\end{remark}

Recall that $x_2= -\mu_1$ is the position of the second mass on the $x$-axis in the original coordinates. Let $R:\mathbb{R}\to \mathbb{R}$ be a function defined as
\begin{equation}
    \label{eq:R-def}
	R(u) := 4u^2 \left( x_{2} + u^{2} \right)^2 + 8\mu_{2} + 8h u^{2} + 8 \mu_{1} \frac{ u^{2} }{ |u^{2}-1| }. 
\end{equation}
 We consider two fixed non-zero real values $d_1,d_2\in \mathbb{R}$ (the $d_1$ and $d_2$ are to be specified later) and define
\[
	\psi_0 : \mathbb{R}^2 \to \Sigma_0\cap\{\Gamma_{2,h}=0\}
\]
with $\Sigma_0 = \{v=0\}$ as
\[
    \psi_0 (z_1, z_2) := (u, 0, p_u, p_v(u,p_u)),
\]
where  
 \begin{align}
    u &= d_1  (z_1 + z_2), \notag \\
    p_u &= d_2  (z_1 - z_2), \label{eq:psi_0-def}\\
    p_v(u,p_u) &= 2u \left( x_2 + u^2 \right)^2 + \sqrt{ R(u) - p_u^2 }.\notag
\end{align}
From (\ref{eq:Gamma_i}) and (\ref{eq:R-def}) it follows that, indeed, $\Gamma_{2,h} \circ \psi_0 = 0$.

The above choice of $\psi_0$ will have two features, which will be important in our construction (see Figure \ref{fig:psi_0_coordinates}). These are summarised in the two lemmas which follow.
\begin{lemma}
    \label{lem:collision-in-psi0}
    In the local coordinates $z=(z_1,z_2)\in \mathbb{R}^2$ given by $\psi_0$, the collision curve (\ref{eq:collision-circle}) lies on the straight line $z_1+z_2=0$.
\end{lemma}
\begin{proof}
    If $z_1+z_2=0$ then $u = d_1  (z_1 + z_2)=0$, so $\psi_0 (z_1, z_2)$ is equal to zero on the $u,v$ coordinates, hence, it is a collision point in the original coordinates.
\end{proof}
\begin{figure}[h!]
\begin{center}
	\begin{overpic}[width=5cm]{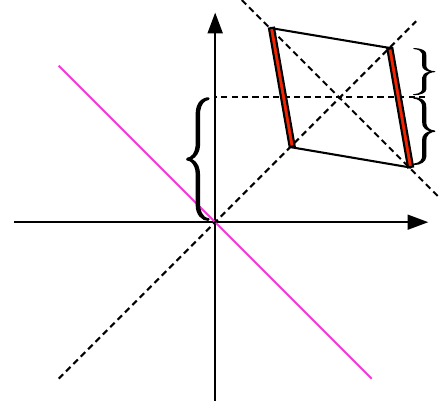}    
        \put (37,58) {$\alpha$}
		\put (100,76) {$\beta(1-L) / (1+L)$}
		\put (100,62) {$\beta$}
		\put (75,88) {$Q$}
		\put (95,37) {$z_1$}
		\put (40,90) {$z_2$}
	\end{overpic}
	\caption{The local coordinates given by $\psi_0$, in which the collision curve is the straight line depicted in pink. Here we also plot a self $S$-symmetric h-set $Q$ from Lemma \ref{lem:Q-self-S-symmetric}.}
    \label{fig:psi_0_coordinates}
\end{center}
\end{figure}
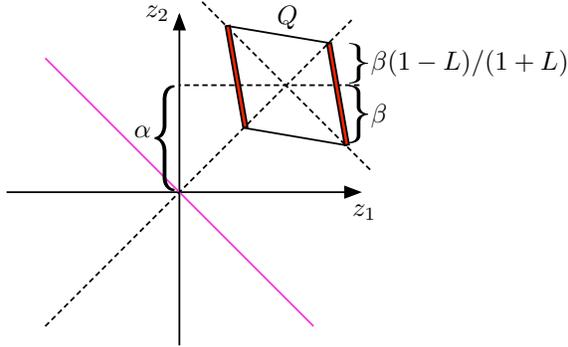
In our construction we will consider h-sets, in the local coordinates given by $\psi_0$, which are additionally rescaled by the linear transformation introduced below. (See the h-set $Q$ in Figure \ref{fig:psi_0_coordinates}.)
Let $L \in [ 0,1 ) $ and consider $\eta_L : \mathbb{R}^{2} \to \mathbb{R}^2$ defined as 
\begin{equation}
    \label{eq:eta-L}
    \eta_L(z)=\frac{1}{1+L}\left( 
    \begin{array}{cc}
        1 & -L \\ 
        -L & 1
        \end{array}
    \right)(z). 
\end{equation}
\begin{remark}
    The reason behind introducing the factor $1/(1+L)$ into the definition of $\eta_L$ is to ensure that $\eta_L (N_c) \subseteq N_c$ and $\eta^{-1}_L = \eta_{-L}$.
\end{remark}
It will turn out to be important for us that h-sets rescaled using $\eta_L$ will be self $S$-symmetric, as described in the following lemma.
\begin{lemma}
    \label{lem:Q-self-S-symmetric}
    Let $\alpha,\beta$ be real numbers. Consider an h-set $Q$ defined by a homeomorphism $c_Q:\mathbb{R}^2\to\mathbb{R}^2$ chosen as (see Figure \ref{fig:psi_0_coordinates})
\[ 
	c_Q(z):=(\alpha, \alpha) +\beta \eta_L(z).
\]
Then the h-set $\mathcal{Q} = (Q,\psi_0)$ is self $S$-symmetric (in the sense of Definition \ref{def:self-S-symmetry}).
\end{lemma}
\begin{proof}
    Recall from Definition \ref{def:NT} that $c_{Q^{T}} = c_Q \circ j.$ We need to show that
    \begin{equation}
        \label{eq:symmetry-target}
        S\circ\psi_{0}\circ c_{Q^{T}}=\psi_{0}\circ c_{Q}.
    \end{equation}
    We notice that $\eta_L \circ j = j \circ \eta_L$ implies that $c_{Q^T} = j \circ c_Q$, which reduces the condition (\ref{eq:symmetry-target}) into:
    \[
        S \circ \psi_0 \circ j = \psi_0.
    \]
    We notice that for $\psi_0 = \left( u, 0, p_u, p_v(u,p_u) \right)$ it is true that:
    \[
        \psi_0 \circ j = \left( u, 0, -p_u, p_v(u,-p_u) \right).
    \]
    The fact that $p_{v}\left(  u,p_{u}\right)  =p_{v}\left(  u,-p_{u}\right)$ concludes the proof.
\end{proof}

\section{Symbolic dynamics in the regularized system}
\label{sec:dynamics_in_regularized_system}
\subsection{Overview}

The first result of this section is Theorem \ref{thm:lyapunov_orbit_energy} where we compute the energy $h_{0}$ of an ejection/collision orbit, which is born from the Lyapunov family, as it ``crashes" with the Earth. If we view the family of Lyapunov orbits in the regularized coordinates, then for all the energies close to and including $h_{0}$ we observe periodic orbits. In the regularized coordinates there is nothing special about the orbit with the energy $h_{0}$ except that it passes through a point, which in the original coordinates corresponds to a collision.

In section \ref{subsec:Covering-relations-along-orbits} we will position a sequence of covering relations, in the regularized coordinates, around the periodic orbit corresponding to the energy $h_{0}$. We will refer to this sequence as a ``collision" sequence, since it involves orbits that come close to a collision point. We will also consider another family of covering relations. This second family is associated with the fact that the stable/unstable manifolds of Lyapunov orbits intersect transversally. In the regularized coordinates such transversal intersection is also present at the energy $h_{0}$ and results in the existence of a homoclinic orbit. Our second sequence of covering relations is associated with this homoclinic excursion. We will refer to this sequence as the ``outer" excursion. We will also show that the two sequences, the ``collision" and the ``outer" can be combined together in arbitrary order, leading to rich symbolic dynamics.

\subsection{Energy of the ejection/collision orbit}
\label{subsec:Energy-of-Lyapunov}

In this subsection we prove the existence of Lyapunov orbit that passes through the regularized collision and we determine the real interval that contains the energy of this orbit.
\begin{theorem}
    \label{thm:lyapunov_orbit_energy}
    There exists positive real value $\tau$ and real value $h_{0}$ such that
    \[
        \Gamma_{2,h_{0}}(w_{0})=0,\quad\Phi_{2\tau}^{2,h_{0}}(w_{0})=w_{0}.
    \]
    where as in (\ref{eq:w0}) $w_{0}=(0,0,0,\sqrt{8\mu_{2}})$ and\footnote{The exact bounds for the value of $h_0$ that we obtained and that we use in the later computations are $h_{0}\in-\big[6404647308743062,6404647308742343\big]\cdot2^{-53}$.}
    \begin{equation}
        \label{eq:h0_energy_interval}
        h_0 \in [ -0.71106, -0.71105 ].
    \end{equation}
\end{theorem}
\begin{proof}
    Let $\Sigma_{v}=\{(u,v,p_{u},p_{v}):v=0\}$. Let us write $P^{h}:\mathbb{R}^{4}\to\Sigma_{v}$ for the Poincar\'e map induced by the flow $\Phi_{s}^{2,h}$. Let us recall that $w_{0}=(0,0,0,\sqrt{8\mu_{2}})$ is a point on the collision circle in the regularized coordinates. 

    Let $\Psi:\mathbb{R}^{5}\to\mathbb{R}^{5}$ be a function defined as,
    \begin{equation}
        \Psi(h,w)=\begin{pmatrix} \Phi_{s_{1}}^{2,h}(w_{0})-w\\ \pi_{3}P^{h}(w) \end{pmatrix},
    \end{equation}
    where $s_{1}\approx0.895629883$. (The particular choice of $s_{1}$ is quite arbitrary and the integration by time $s_{1}$ here is only due to technical reasons: We perform a ``parallel shooting" type method since the integration time from $w_{0}$ to $\Sigma_{v}$ is rather long. Had we considered looking for zeros of $h\mapsto\pi_{3}P^{h}(w_{0})$ the method would still work, but be less accurate than we would need; see also Remark \ref{rem:s1_value}). Observe that if we show that for some $h_{0}\in\mathbb{R}$ and $w_{1}\in\mathbb{R}^{4}$ we have $\Psi(h_{0},w_{1})=0$, then we will obtain that $w_{2}:=P^{h_0}(w_{1})=P^{h_{0}}(\Phi_{s_{1}}^{2,h_{0}}(w_{0}))$ is a self $S$-symmetric point, where $S$ is defined in (\ref{eq:S_symmetry}). This is because $w_{2}$ is zero on the coordinate $v$ due to the choice of the section $\Sigma_{v}$ and also zero on the coordinate $p_{u}$ since $\pi_{3}w_{2}=P^{h_0}(w_{1})=0$. We know that $w_{2}=P^{h_{0}}(\Phi_{s_{1}}^{2,h_{0}}(w_{0}))$, so letting $\tau$ be the time for which $w_{2}=\Phi_{\tau}^{2,h_{0}}(w_{0})$ we obtain
    \begin{align*}
        w_{0} & =\Phi_{-\tau}^{2,h_{0}}(w_{2})=\Phi_{-\tau}^{2,h_{0}}(S(w_{2}))=S(\Phi_{\tau}^{2,h_{0}}(w_{2}))=S(w_{0})=w_{0},
    \end{align*}
    so we have a periodic orbit passing through the collision point $w_{0}$ in the regularized coordinates, which means that we have an ejection/collision orbit in the original coordinates. 

    We have performed an interval arithmetic validation by means of the interval Newton Theorem \ref{thm:extended_interval_Newton_theorem}, which proved that function $\Psi$ has a unique root for $h_{0}$ that satisfies (\ref{eq:h0_energy_interval}). For our computer assisted proof we have used the CAPD\footnote{Computer Assisted Proofs in Dynamics: http://capd.ii.uj.edu.pl} package \cite{MR4283203,MR4395996}. The computations took less than a minute. This concludes our proof.
\end{proof}

\begin{remark}
    \label{rem:s1_value}
    The exact value of $s_1$ is $58696/65536$. Such value is chosen to minimize the diameter of the estimation of value $h_0$. With the value present in the proof, the size of the exact bounds of $h_0$ is approximately $7.98 \cdot {10}^{-14}$. For comparison, with $s_1$ equal to $1/2$ the size is $2.75 \cdot {10}^{-13}$.
\end{remark}

\subsection{Covering relations for excursions around the ejection/collision orbit and its homoclinic}
\label{subsec:Covering-relations-along-orbits}

Throughout this section we work with the regularized flow $\Phi_{s}^{2,h_{0}}$ generated by the Hamiltonian $\Gamma_{2,h_{0}}$, which regularizes the collision with the second mass, which in the Earth-Moon system is the Earth. Here we choose $h_{0}$ to be the value from Theorem \ref{lem:inverse-psi}. 

We will consider a sequence of points $w_k\in \mathbb{R}^4$ for  $k \in \{ 0, ..., K \}$ with $K = 18$. The choice of the points $w_0,\ldots, w_K$ is written out in Table \ref{table:wk} in \ref{appendix:coordinate_systems}. The points are also depicted in Figures \ref{fig:homoclinic_reg_orbit} and \ref{fig:homoclinic_std_orbit}. At these points we will position consecutive Poincar\'e sections $\Sigma_k$ equipped with local coordinate changes 
\[
    \psi_{k}:\mathbb{R}^{2}\to\Sigma_{k}\cap\{ \Gamma_{2,h_{0}} = 0\}, \qquad \mbox{for }k=0,\ldots,K.
\]

The first of the points $w_0$ will be chosen as (\ref{eq:w0}), and at this point we will choose the section $\Sigma_0=\{v=0\}$ and the parametrisation $\psi_0$, which was introduced in section \ref{sec:section-at-collision}. The choice of the constants $d_1$ and $d_2$ defining $\psi_0$ is discussed in \ref{appendix:coordinate_systems} and written out in (\ref{eq:d1-d2-choice}). At the remaining points $w_1,\ldots,w_K$ we will consider Poincer\'e sections described in section \ref{sec:Sections-on-energy}. The choice of the vectors $\hat u_k$ and $\hat s_k$ specifying the coordinate changes $\psi_k$, for $k=1,\ldots,K$ is written out in \ref{appendix:coordinate_systems} (see in particular Tables \ref{table:uk} and \ref{table:sk}). Below we describe the method we used to reach such choice.

\begin{figure}[h!]
    \begin{center}
    \begin{overpic}[width=6cm]{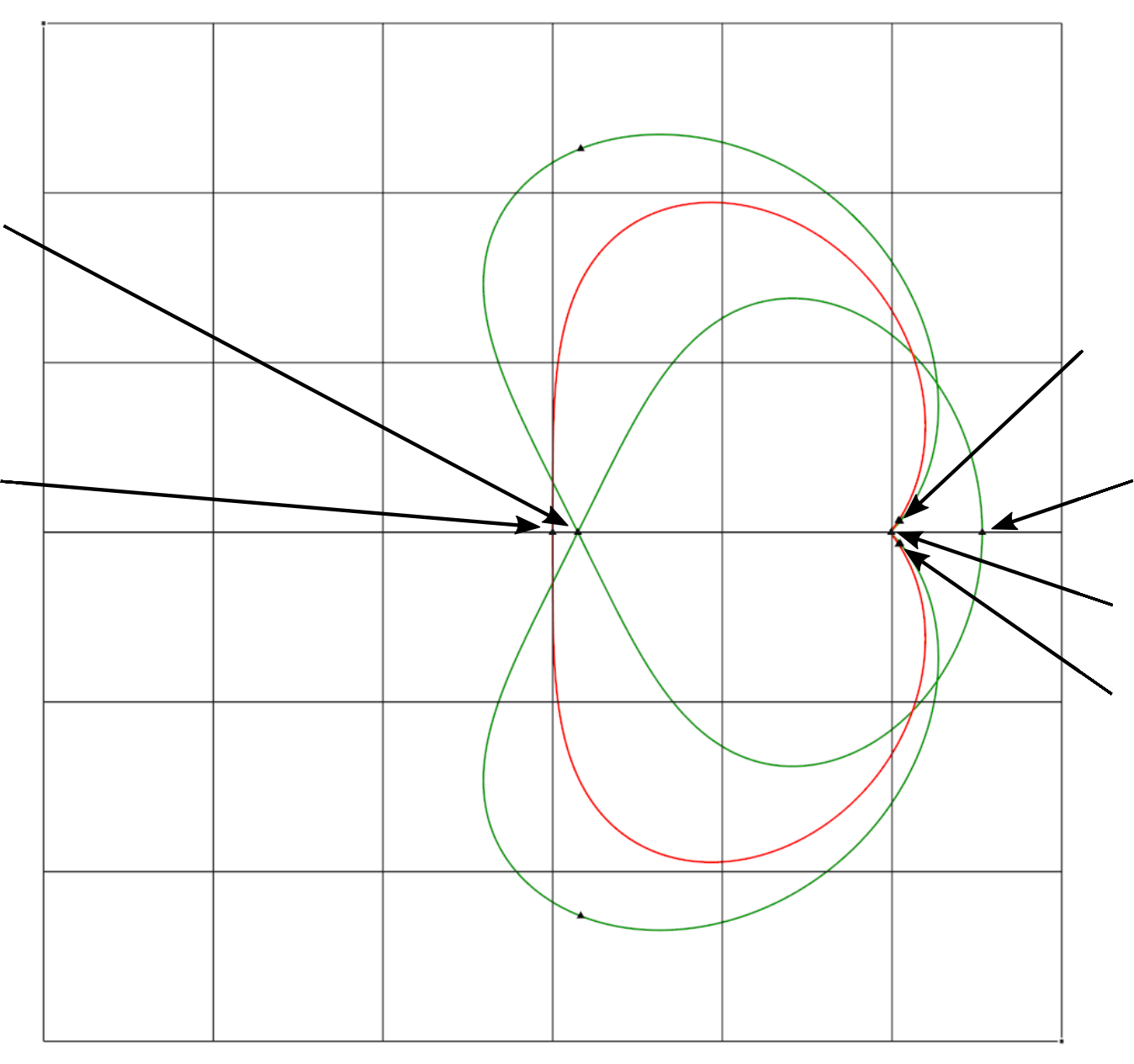}    
    \put (-10,53) {$w_0, w_4, w_8, w_{12}$}
    \put (95,65) {$w_1, w_5, w_9, w_{13}$}
    \put (95,38) {$w_2, w_6, w_{10}, w_{14}$}
    \put (95,30) {$w_3, w_7, w_{11}, w_{15}$}
    \put (50,10) {$w_{16}$}
    \put (-10, 72) {$w_{17}$}
    \put (100,52) {$w_{18}$}
    \put (45,83) {$S(w_{16})$}
    \put (13,-5) {$-1$}
    \put (47,-5) {$u$}
    \put (77,-5) {$1$}
    \put (-2,77) {$1$}
    \put (-3,45) {$v$}
    \put (-7,15) {$-1$}
    \end{overpic}
    \end{center}
    \caption{The collision/ejection orbit (red) and its homoclinic (green) in the regularized coordinates (projection on coordinates $u, v$). }
    \label{fig:homoclinic_reg_orbit}
\end{figure}
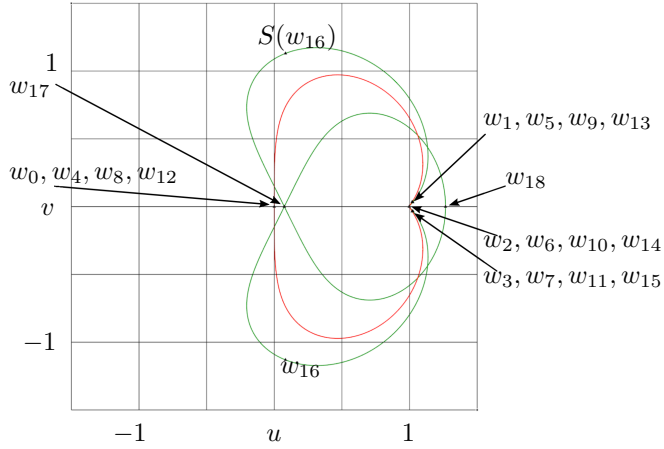

The points $w_0$, $w_1$ and $w_2$ are the same as the respective points in the proof of Theorem \ref{thm:lyapunov_orbit_energy}. We recall that:
\begin{equation}
    \label{eq:w0_w2_S_symmetry}
    w_0 = S(w_0), \quad w_2 = S(w_2).
\end{equation}
We also define the point $w_3$ as
\begin{equation}
    \label{eq:w1_w3_S_symmetry}
    w_3 := S(w_1).
\end{equation}
We see that points $w_0, ..., w_3$ lie on the ejection/collision orbit.
We place the remaining points $w_4, ..., w_K$  approximately on its homoclinic. We choose $w_K$ so that
\begin{equation}
    \label{eq:wK_S_symmetry}
    w_K = S(w_K),
\end{equation}
At the points $w_0, ..., w_K$ we position Poincar\'e sections defined as (\ref{eq:section-choice}), and define 
\[P_k:\mathbb{R}^4\to \Sigma_k, \quad \mbox{for }k=0,\ldots,K,\] 
to be maps to the corresponding sections along the flow $\Phi_s^{2,h_0}$. 
Since $w_0, ..., w_3$ lie on the ejection/collision orbit we have
\begin{equation}\label{eq:points-on-orbit}
w_1=P_1(w_0),\qquad w_2=P_2(w_1),\qquad w_3=P_3(w_2),\qquad w_0=P_0(w_3),
\end{equation}
and $w_0$ is a fixed point of  $P_0 \circ P_3 \circ P_2 \circ P_1$. 

To introduce parametrizations $\psi_k$, for $k=1,\ldots,K$ (\ref{eq:local-parametrization-choice}) on $\Sigma_k \cap \{\Gamma_{2,h_0}=0\}$ we choose appropriate $\hat u_k$ and $\hat s_k$; see Remark \ref{rem:local-map-choice}. 
The vector $\hat{u}_0$ is a unit eigenvector associated to the largest eigenvalue of the matrix
\[
    D \Phi_{2\tau}^{2,h_{0}} (w_0),
\]
where the eigenvalue problem is solved numerically. This vector is also propagated numerically around the periodic orbit to produce $\hat{u}_k$ at $w_k$ for $k=1,2,3$. We also choose $\hat{u}_4 := \hat{u}_0$ and propagate it around the homoclinic orbit to produce $\hat{u}_k$ at $w_k$ for $k=5,\ldots,K$.
The vectors $\hat{s}_0, ..., \hat{s}_K$ are obtained in a similar way and from using the $S$-symmetry of the system. In particular, in our argument we will explicitly use the fact that we choose
\begin{equation}
    \label{eq:S_backsymmetric_stable_directions}
    \quad \hat{s}_2 = S \hat{u}_2, \quad \hat{s}_K = S \hat{u}_K,
\end{equation}
\begin{equation}
    \label{eq:s1_u3_S_backsymmetry}
    \hat{s}_1 = S \hat{u}_3 \quad \hat{s}_3 = S \hat{u}_1.
\end{equation}
(See Figure \ref{fig:homoclinic_std_orbit} for how the points $w_k$ corresponding to (\ref{eq:S_backsymmetric_stable_directions}) and (\ref{eq:s1_u3_S_backsymmetry}) are positioned.)

\begin{figure}[h]
    \begin{center}
    \begin{overpic}[width=6cm]{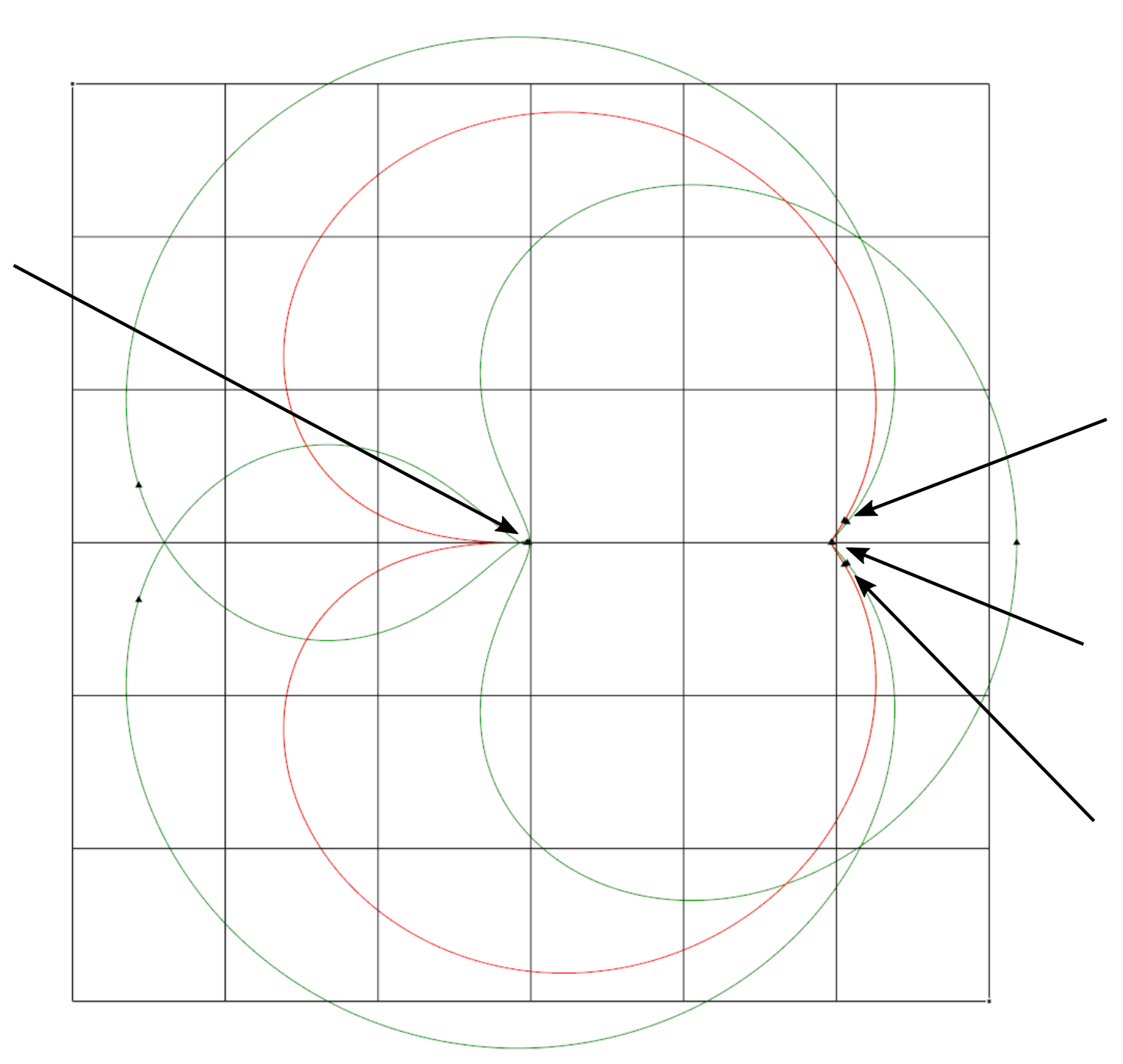}    
        \put (-15,73) {$w_0, w_4, w_8, w_{12}, w_{17}$}
        \put (95,59) {$w_1, w_5, w_9, w_{13}$}
        \put (95,35) {$w_2, w_6, w_{10}, w_{14}$}
        \put (95,19) {$w_3, w_7, w_{11}, w_{15}$}
        \put (7,44) {$w_{16}$}
        \put (92,46) {$w_{18}$}
        \put (3,54) {$S(w_{16})$}
        \put (14,-3) {$-1$}
   	    \put (47,-3) {$x$}
    	\put (73,-3) {$1$}
    	\put (-3,45) {$y$}
   	    \put (-5,17) {$-1$}
    \end{overpic}
    \end{center}
    \caption{The collision/ejection orbit (red) and the homoclinic orbit (green) in standard coordinates. }
    \label{fig:homoclinic_std_orbit}
\end{figure}
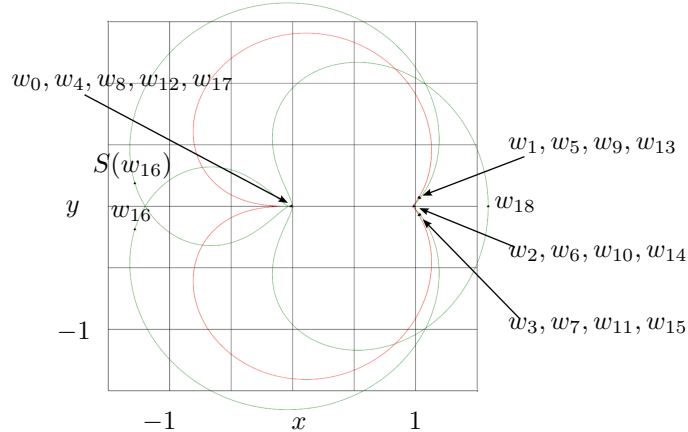
The vectors $\hat u_1,\ldots,\hat u_K$ and $\hat s_1,\ldots,\hat s_K$ determine the local coordinate changes $\psi_1,\ldots,\psi_K$. The vectors also $\hat u_0$ and $\hat s_0$  determine the choice of the constants $d_1$ and $d_2$ which define $\psi_0$ (\ref{eq:psi_0-def}). This choice is discussed in (\ref{eq:why-d1-d2}) and (\ref{eq:d1-d2-choice}) in \ref{appendix:coordinate_systems}.

For $k=0,\ldots,K$ we define a sequence of h-sets $\mathcal{N}_k$ on the sections $\Sigma_k \cap \{\Gamma_{2,h_0}=0\}$, which are parameterized by $\psi_k$, namely:
\[
    \mathcal{N}_k = (N_c, \psi_k),    
\]
where $N_c = [-1,1]^2$. Our choices of the vectors $\hat{u}_k, \hat{s}_k$ defining the parametrizations are appropriately rescaled so that we can choose all of these h-sets to be of the same size in the local coordinates at $w_k$.

\begin{remark}
    By Lemma \ref{lem:Q-self-S-symmetric} (taking $\alpha=L=0$ and $\beta=1$), the h-set $\mathcal{N}_0$ is self $S$-symmetric. Also, due to (\ref{eq:w0_w2_S_symmetry}), (\ref{eq:wK_S_symmetry}) and (\ref{eq:S_backsymmetric_stable_directions}) the h-sets $\mathcal{N}_2$ and $\mathcal{N}_K$ are self $S$-symmetric. Furthermore, due to (\ref{eq:w1_w3_S_symmetry}) and (\ref{eq:s1_u3_S_backsymmetry}) we have  $\mathcal{N}_3=S\mathcal{N}_1$. 
\end{remark}

For the result presented below consider a sequence of Poincar\'e maps 
\begin{equation} \label{eq:PS-maps}
P^S_k:S\Sigma_k \to S\Sigma_{k-1}  \qquad \mbox{for } k=1,\ldots,K.
\end{equation}

\begin{theorem}
    \label{thm:covering_relations_1_3}
    We have the following coverings (here we use the short-hand notation (\ref{eq:sec-covering-1}--\ref{eq:sec-covering-2})):
    \begin{align}
        \label{eq:covering_sequence_c1}
        &\mathcal{N}_0 \xRightarrow{P_1}
        \mathcal{N}_1 \xRightarrow{P_2}
        \mathcal{N}_2, \\
        \label{eq:covering_sequence_c2}
        &\mathcal{N}_2 \xLeftarrow{P_3}
        \mathcal{N}_3 \xLeftarrow{P_0}
        \mathcal{N}_0,
    \end{align}
    \begin{align}
        \label{eq:covering_sequence_o1}
        &\mathcal{N}_0 \xRightarrow{P_1}
        \mathcal{N}_1 \xRightarrow{P_2} 
        \mathcal{N}_2 \xLeftarrow{P_3} 
        \mathcal{N}_3 \xRightarrow{P_4} ... \xRightarrow{P_K}
        \mathcal{N}_K, \\
        \label{eq:covering_sequence_o2}
        &\mathcal{N}_K \xLeftarrow{P^S_K}
        S \mathcal{N}_{K-1} \xLeftarrow{P^S_{K-1}} ... \xLeftarrow{P^S_4}
        S \mathcal{N}_3 \xRightarrow{P^S_3}
        S \mathcal{N}_2 \xLeftarrow{P^S_2}
        S \mathcal{N}_1 \xLeftarrow{P^S_1}
        \mathcal{N}_0.
     \end{align}
\end{theorem}
\begin{proof}
    Existence of the covering relations (\ref{eq:covering_sequence_c1}) was proven by direct rigorous numeric validation performed with the CAPD package \cite{MR4283203,MR4395996}. Our local coordinates have been chosen in a way that the Poincar\'e maps expressed in the local coordinates are well aligned with the dynamics (see section \ref{sec:Sections-on-energy}). The CAPD package allows for the computation of Poincar\'e maps $P_{k}$. As discussed in section \ref{sec:Sections-on-energy}, the coordinate changes $\psi_{k}$ and $\psi_{k}^{-1}$ are also computable in interval arithmetic. As a result we can compute the interval enclosures of the Poincar\'e maps in the local coordinates and validate the covering relation conditions. 
    
    By applying Lemma \ref{lem:S_backsymmetry} from (\ref{eq:covering_sequence_c1}) we obtain that
    \[
        S\mathcal{N}_2 \xLeftarrow{P^S_2}
        S\mathcal{N}_1 \xLeftarrow{P^S_1}
        S\mathcal{N}_0.
    \]
    Since $\mathcal{N}_0$ and $\mathcal{N}_2$ are self $S$-symmetric and $\mathcal{N}_3=S\mathcal{N}_1$ by Lemma \ref{lem:self-symmetric-covering} we see that the above sequence of coverings is the sequence (\ref{eq:covering_sequence_c2}), only written in different notation, so we have established (\ref{eq:covering_sequence_c2}).

    We prove the existence of the covering relations $\mathcal{N}_3 \xRightarrow{P_4} ... \xRightarrow{P_K}
    \mathcal{N}_K$ with the direct rigorous numeric validation performed with CAPD package. This, together with (\ref{eq:covering_sequence_c1}) and (\ref{eq:covering_sequence_c2}), establishes (\ref{eq:covering_sequence_o1}). The sequence of covering relations (\ref{eq:covering_sequence_o2}) follows directly from (\ref{eq:covering_sequence_o1}) by applying Lemmas \ref{lem:S_backsymmetry} and \ref{lem:self-symmetric-covering}, since $\mathcal{N}_0$ and $\mathcal{N}_K$ are  self $S$-symmetric.

    The computer assisted validation of the conditions takes less than 5 minutes build time and execution time in total. The code is available at \cite{pcr3bp_code}. 
\end{proof}

\begin{corollary}
    \label{col:orbits_in_reg}
    For an arbitrary choice of sequence from $\{c,o\}^{\mathbb{Z}}$ there exists an orbit which behaves according to the rule:
    \begin{itemize}
        \item starts from $\mathcal{N}_0$, does not visit $\mathcal{N}_K$ and ends in $\mathcal{N}_0$ for symbol $c$,
        \item starts from $\mathcal{N}_0$, visits $\mathcal{N}_K$ and ends in $\mathcal{N}_0$ for symbol $o$.
    \end{itemize}
    Moreover, for any $n\in\mathbb{N}$, and for any choice of a finite sequence $\{c,o\}^{n}$, there exists a periodic orbit which behaves in the accordance to the above rule.
\end{corollary}
\begin{proof}
    The result follows by combining Theorems \ref{thm:generic_covering_relations} and \ref{thm:covering_relations_1_3}. We note that the covering relation sequences (\ref{eq:covering_sequence_c1}--\ref{eq:covering_sequence_c2}) and (\ref{eq:covering_sequence_o1}--\ref{eq:covering_sequence_o2}) can be glued in any prescribed order. Moreover, we see that the excursion which shadows (\ref{eq:covering_sequence_c1}--\ref{eq:covering_sequence_c2}) will not visit  $\mathcal{N}_K$, and the excursion which shadows (\ref{eq:covering_sequence_o1}--\ref{eq:covering_sequence_o2}) will visit $\mathcal{N}_K$.
\end{proof}

In the above corollary we have chosen the index ``c" to stand for ``collision". We have done so due to the fact that $\mathcal{N}_0$ is positioned at the a collision point $w_0$ of the ejection/collision orbit. We have chosen the index ``o" to stand for ``outer" dynamics. 

We finish by observing that the collision circle (\ref{eq:collision-circle}) is contained in $\Sigma_0$. The circle expressed in the local coordinates given by $\psi_0$ intersected with $\mathcal{N}_0$ produces a line passing through the origin; see Lemma \ref{lem:collision-in-psi0}.
As a consequence, the orbits from Corollary \ref{col:orbits_in_reg} might involve collisions at $\mathcal{N}_0$. In the next section we take a closer look at what happens on $\mathcal{N}_0$ which will allow us to ensure that orbits from Corollary \ref{col:orbits_in_reg} can avoid colliding with Earth.

\section{Collision approach\label{sec:collision-approach}}
In this section we extend the result proven in section \ref{sec:dynamics_in_regularized_system} by ensuring that the orbits which follow the symbolic dynamics avoid collisions. Furthermore, we show that such orbits can approach arbitrarily close to collision and then move away from it to a prescribed distance.

The idea for our construction is as follows. We will construct sequences of covering relations from $\mathcal{N}_K$ to $\mathcal{Q}_k$, where $\mathcal{Q}_k$ will be a self $S$-symmetric h-set inside of $ \mathcal{N}_0$, for $k\in \mathbb{N}$. Here $\mathcal{N}_K$ and $\mathcal{N}_0$ are the h-sets involved in the coverings (\ref{eq:covering_sequence_c1}--\ref{eq:covering_sequence_o2}) and $\mathcal{Q}_k$ will be defined later on. From our construction it will follow that the higher the choice of $k$, the more turns are made around the ejection/collision orbit. With each turn we will approach closer to a collision. See Figure \ref{fig:overview1} for a cartoon depicting the intuition. We will be careful though to construct additional h-sets inside of $\mathcal{N}_0$ so that we are sure that we avoid collisions.  The idea is depicted in a cartoon in Figure \ref{fig:overview2}. (More true and careful plots of a single step of the construction are also in Figures \ref{fig:lemma-2} and \ref{fig:lemma-3}.) Since both of the sets $\mathcal{N}_K$ and $\mathcal{Q}_k$ will be self $S$-symmetric, we will automatically obtain a sequence of coverings in the other direction: from $\mathcal{Q}_k$ to $\mathcal{N}_K$. The proof of Theorem \ref{thm:primary_result} will follow from appropriate compositions of such sequences.

\begin{figure}
\begin{center}
	\begin{overpic}[width=5cm]{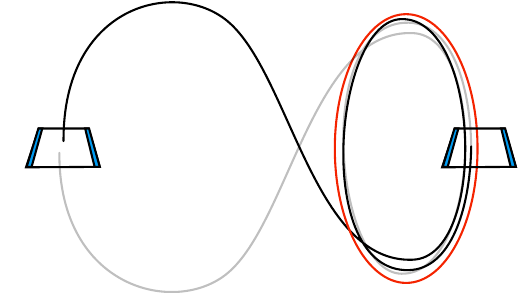}
		\put (99,30) {$\mathcal{N}_0$}
		\put (-5,30) {$\mathcal{N}_K$}
	\end{overpic}
\end{center}
\caption{With each turn around the ejection/collision orbit (depicted in red) we come closer to the collision. In black we depict an orbit which makes two such turns, i.e. passes through $\mathcal{N}_0$ twice. Existence of such orbits will follow from appropriate sequences of covering relations. By the $S$-symmetry of the system we will obtain $S$-symmetric sequences of covering relations, which will allow us to come back to $\mathcal{N}_K$. A resulting returning orbit is depicted in grey. \label{fig:overview1}}
\end{figure} 
\begin{figure}
\begin{center}
	\begin{overpic}[width=4.0cm]{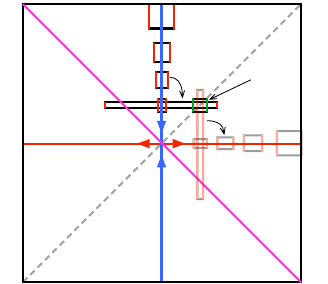}
	\put (95,20) {$\mathcal{N}_0$}
	\put (78,62) {$\mathcal{Q}$}
	\put (25,70) {$C$}
	\end{overpic}
\end{center}
\caption{A sequence of h-sets which consecutively cover one another is constructed along the stable manifold (in blue). We also have an self $S$-symmetric h-set $\mathcal{Q}$ involved in the sequence of coverings. The pink line $C$ is the set of collision points in $\mathcal{N}_0$. From the $S$-symmetry of the system there automatically exists a matching sequence of $S$-symmetric h-sets along the unstable manifold (in red). In total, this produces a sequence of h-sets which approach and depart from the collision; all of the h-sets along the stable/unstable manifolds, together with the h-set $\mathcal{Q}$, are disjoint the collision line $C$. From our construction it will follow that we can make such sequence to be of arbitrary length.  \label{fig:overview2}}
\end{figure} 

We start our discussion with an abstract result for a hyperbolic fixed point at the origin on the plane, where we show a construction of successive covering relations which approach the origin but stay in the first quarter of the plane. This is done in section \ref{sec:cover-fix-pt}. (This will be useful for us since it will turn out that the collision curve in $\mathcal{N}_0$ stays in the second and in the fourth quarter of the plane; as depicted in Figure \ref{fig:overview2}.) Then, in section \ref{sec:main-proof}, we apply the method result to prove Theorem \ref{thm:primary_result}.

\subsection{Covering relations in a neighbourhood of a hyperbolic fixed point\label{sec:cover-fix-pt}}
In this section we present an explicit construction of a sequence of covering relations that occur in a neighborhood of a hyperbolic fixed point on the plane. The trajectories shadowing that sequence of covering relations approach arbitrarily close the origin and remain in the first quarter of the plane. 

Recall that $N_c=[-1,1]^2\subset \mathbb{R}^2$. Let $f =(f_1,f_2): N_c \to \mathbb{R}^2$ be a  $C^1$ function such that $f(0) = 0$. We assume that there exist positive real numbers $\alpha, \beta, \rho, c$ such that the partial derivatives of $f$ satisfy:
\begin{equation}
    \label{eq:der_cond_1}
    \bigg[\frac{\partial f_{1}}{\partial z_1} (N_c) \bigg] > \alpha, \quad
    \bigg[\frac{\partial f_{1}}{\partial z_2} (N_c) \bigg] \subset (-c,0),
\end{equation}
\begin{equation}
    \label{eq:der_cond_2}
    \bigg[\frac{\partial f_{2}}{\partial z_1} (N_c) \bigg] \subset (0,c), \quad
    \bigg[\frac{\partial f_{2}}{\partial z_2} (N_c) \bigg] \subset \big( \beta, \rho \big),
\end{equation}
and $\rho > \beta$.

We define h-sets: 
\[
    R_{a,b} = [0,b] \times [a,b]
\]
and
\[
    Q_{a,b} = [a,b]\times [a,b].
\]
See Figure \ref{fig:lemma-2} for the depiction of these sets and for a visualization of the following lemma.
\begin{lemma}
    \label{lem:cov_rel_3}
    If $\alpha > 2c+\rho$ then for any real values $a$, $b$ such that $0 < a < b < 1$ and for $a',b'\in \mathbb{R}$ chosen as
    \begin{equation}
        a' = a \cdot \beta, \quad b' = \big(c + \rho \big) \cdot b,
    \end{equation}
    we have
    \begin{equation}
        R_{a,b} \xRightarrow{f} R_{a',b'} \quad\text{and}\quad
        R_{a,b} \xRightarrow{f} Q_{a',b'}.
    \end{equation}
\end{lemma}
\begin{proof}
    Since $f(0) = 0$, by the mean value theorem we see that:
    \begin{equation}
        \label{eq:mean_value_theorem}
        f(z) \subset Df(N_c) z
    \end{equation}
    for every $z$ that belongs to $N_c$. We notice that:
    \begin{align}
        \begin{aligned}
        \pi_2 f\big( R_{a,b} \big)
            &\subset [0,b]\cdot (0,c) + [a,b] \cdot \big( \beta, \rho \big) \\
            &\subset \big( a \cdot \beta, (c + \rho) \cdot b \big) \\
            &\subset (a',b')
        \end{aligned}
    \end{align}
    which asserts the contraction condition for both covering relations. Futhermore, we notice that:
    \begin{align}
        \begin{aligned}
            \pi_1 f\big( R^{l}_{a,b} \big) < 0\cdot [0,b] + (-c,0)\cdot [a,b] < 0
        \end{aligned}
    \end{align}
    which asserts the left expansion condition for both covering relations. Finally, we notice that:
    \begin{align}
        \begin{aligned}
            \pi_1 f\big( R^{r}_{a,b} \big)
            &> \alpha b + (-c,0) \cdot [a,b] \\
            &> (\alpha -c)b \\
            &> (c + \rho)b \\
            &> b',
        \end{aligned}
    \end{align}
    which asserts the right expansion condition for both covering relations and concludes the proof.
\end{proof}

The bounds (\ref{eq:der_cond_1}--\ref{eq:der_cond_2}) and the assumptions of Lemma \ref{lem:cov_rel_3} are based on the fact that the stable and unstable manifolds of the fixed point are not tangent to the axes (see Figure \ref{fig:lemma-2}).

\begin{figure}[h]
    \begin{center}
    	\begin{overpic}[width=8cm]{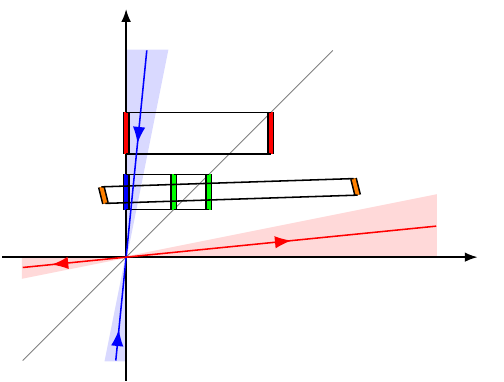}
		\put (95,27){$z_1$}
		\put (27,75){$z_2$}
		\put (15,52){$R_{a,b}$}
		\put (60,42){$f(R_{a,b})$}
		\put (14,42){$R_{a',b'}$}
		\put (38,30){$Q_{a',b'}$}
   	\end{overpic}
    \end{center}
    \caption{A schematic picture of an h-set $R_{a,b}$ (red), its image under $f$ (orange), an h-set $R_{a',b'}$ (blue) and an h-set $Q_{a',b'}$ (green). Shaded blue and red areas represent bounds for the stable and unstable manifolds respectively. }
    \label{fig:lemma-2}
\end{figure}
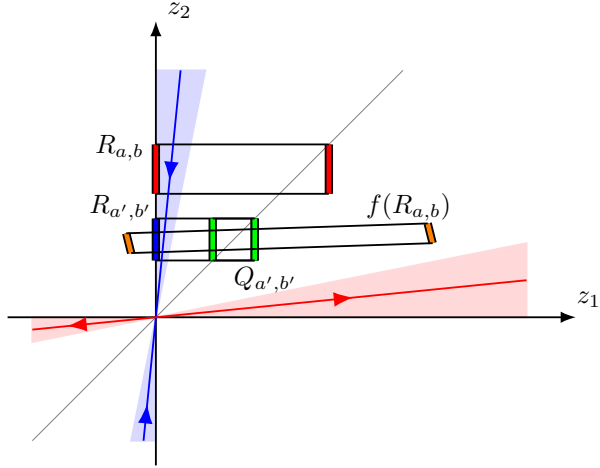

Consider now a map $g$, for which the coordinates are chosen so that the stable and unstable manifolds are aligned with the axes. We can then define
\begin{equation}
    \label{eq:f-def-avoid-collision}
    f=\eta_L^{-1}\circ g\circ \eta_L
\end{equation}
(see (\ref{eq:eta-L}) for the definition of $\eta_L$).

In such case, for a given map $g$ and for a chosen constant $L$, we can check (in interval arithmetic) that (\ref{eq:der_cond_1}--\ref{eq:der_cond_2}) is fulfilled by $f$ (see Figure \ref{fig:eta_L_diagram}).

\begin{figure}[h]
    \begin{center}
    	\begin{overpic}[width=5cm]{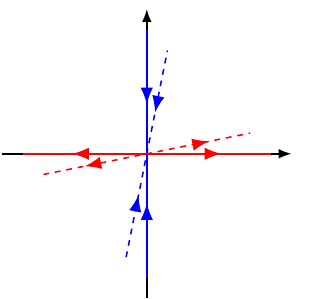}
	\put (90,50){$z_1$}
	\put (50,90){$z_2$}
	\end{overpic}
    \end{center}
    \caption{Stable and unstable manifolds of the map $g$ (continuous blue and red lines, respectively) and stable and unstable manifolds of the map $f$ for the parameter $L=1/5$ (dashed blue and red lines, respectively). }
    \label{fig:eta_L_diagram}
\end{figure}

\begin{figure}[h]
    \begin{center}
    \begin{overpic}[width=6cm]{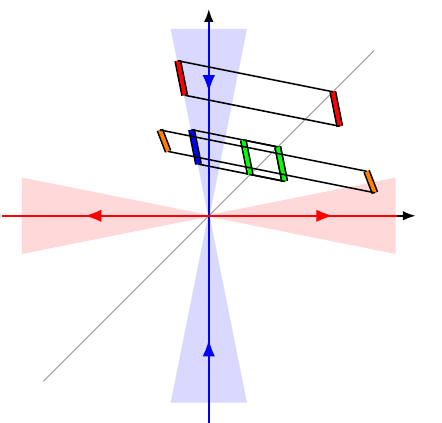}
        \put (93,50){$z_1$}
        \put (50,93){$z_2$}
	\end{overpic}
    \end{center}
    \caption{Schematic picture of h-set $R_k$ (red), its image under $g$ (orange), h-set $R_{k+1}$ (blue) and h-set $Q_{k+1}$ (green). Shaded blue and red areas represent rigorous bounds for the stable and unstable manifolds of the maps $g_i$ respectively. }
    \label{fig:lemma-3}
\end{figure}

Consider fixed real values $a,b \in (0,1)$ such that $a<b$ and consider the following h-sets (see Figure \ref{fig:lemma-3}):
\begin{equation}
    R_{k+1} := \eta_L R_{\beta^k a, ( c+\rho )^k b}, \quad
    Q_{k+1} := \eta_L Q_{\beta^k a, ( c+\rho )^k b} \quad \text{for }k\in \mathbb{N}. \label{eq:Rk-Qk-def}
\end{equation}
 These h-sets are in the coordinates of the map $g$. (We index the sets in (\ref{eq:Rk-Qk-def}) somewhat artificially by having $k+1$ on the left. This is done intentionally to start with $R_{1}, Q_{1}$. Such numbering will then match well with the numbering of Poincar\'e sections on which we will position such sets in the construction for the proof of Theorem \ref{thm:primary_result}.)
 
\begin{lemma}
    \label{lem:gk-approach}
    Assume that for $f$ defined in (\ref{eq:f-def-avoid-collision}) we have the bounds (\ref{eq:der_cond_1}--\ref{eq:der_cond_2}) and that $\alpha >2c+\rho $ and $c+\rho <1$.
    Then for every $k\in \mathbb{N}$ 
    \begin{align}
        & R_{k} \xRightarrow{g} R_{k+1},  \label{eq:RR-G-implies} \\
        & R_{k} \xRightarrow{g} Q_{k+1}.  \label{eq:RQ-G-implies}
    \end{align}
\end{lemma}
\begin{proof}
    The result is a direct consequence of Lemma \ref{lem:cov_rel_3}.
\end{proof}

\subsection{Proof of Theorem \ref{thm:primary_result}\label{sec:main-proof}}

In this section we prove Theorem \ref{thm:primary_result}. We do so by first formulating a number of auxiliary lemmas. The first two concern how collisions relate to the h-sets which are involved in the coverings (\ref{eq:covering_sequence_c1}--\ref{eq:covering_sequence_c2}). We then complement the sequences of coverings (\ref{eq:covering_sequence_o1}--\ref{eq:covering_sequence_o2}) with additional h-sets which approach a collision point, but which are disjoint from the collision set. We do this by using the techniques from section \ref{sec:cover-fix-pt}. The proof of Theorem \ref{thm:primary_result} is then given at the end of this section.

Let $C:\mathbb{R}^4\to \mathbb{R}^3$ be a function defined as
\[
    C(u,v,p_u,p_v) := \left( u, v, p_u^2 + p_v^2 - 8\mu_2 \right).
\]
By (\ref{eq:collision-circle}) we see that  $w \in \mathbb{R}^4$ satisfying $C(w) = 0$ correspond to a collision with the regularized mass, which is the Earth in our case. 
\begin{lemma} 
    \label{lem:collision_avoidance_1} 
    Any orbit which is a result of shadowing the sequences of covering relations (\ref{eq:covering_sequence_c1}--\ref{eq:covering_sequence_o2})  can collide with Earth only at the set $\psi_0(\mathcal{N}_0)$.
\end{lemma}
\begin{proof}
    The proof follows from direct numerical computations performed in CAPD. The CAPD library provides the functionality of obtaining bounds for orbit enclosures. (This is achieved with the use of the {\tt SolutionCurve} component of the library.) We have therefore validated that $C\ne 0$ for the entire bound on the orbits starting at $\psi_k(\mathcal{N}_k)\subset \Sigma_k$ and ending at $\Sigma_{k+1}$ for $k=1,\ldots,K-1$. The set $\{C=0\}$ is contained in $\Sigma_0$. We have checked that orbits starting from $\psi_0(\mathcal{N}_0)\subset \Sigma_0$ immediately leave $\Sigma_0$ and do not revisit it while crossing to $\Sigma_1$. This way we have ensured in our computer assisted program that the orbits resulting from the sequences (\ref{eq:covering_sequence_c1}), (\ref{eq:covering_sequence_o1}) can only meet the collision at $\psi_0(\mathcal{N}_0)$. The same fact for (\ref{eq:covering_sequence_c2}), (\ref{eq:covering_sequence_o2}) follows from the symmetry of the system.
\end{proof}
In our proof we will be using the method described in section \ref{sec:cover-fix-pt}. In particular, we will use the change of coordinates $\eta_L$ from (\ref{eq:eta-L}) and the choice of the parameter that turned out to be suitable for us is $L = 0.0039.$
\begin{lemma}
    \label{lem:collision_avoidance_3}
    The image of the collision set $\{C=0\}$ in the local coordinates given by $\psi_0\circ \eta_L$, i.e. $(\psi_0\circ \eta_L)^{-1}(\{C=0\})\subset \mathbb{R}^2$ intersects with the box $[0,1]^2\subset \mathbb{R}^2$ at a single point, which is the origin.
\end{lemma}
\begin{proof}
    This follows directly from Lemma \ref{lem:collision-in-psi0} and (\ref{eq:eta-L}).
\end{proof}

Let us define h-sets on sections (according to definition \ref{def:h_set_on_section}):
\[
    \mathcal{R}_k = ( R_k, \psi_{k \mod 4} ),
\]
\[
    \mathcal{Q}_k = ( Q_k, \psi_{k \mod 4} ),
\]
where the constants used $a,b,c,\alpha,\beta,\rho$ in these sets are chosen by us as\footnote{Such choice of constants was dictated by trial and error and leads to the desired properties.}
\[
    a=\frac{1}{256},\quad b=\frac{255}{256},\quad c=0.021, \quad
    \alpha = 5.09,\quad \beta=0.195,\quad \rho=0.197.
\]
Let us notice, that these h-sets will be consecutively positioned on $\Sigma_{0},\Sigma_{1},\Sigma_{2}$ and $\Sigma_{3}$. 

By Lemma \ref{lem:Q-self-S-symmetric}, for every $k\in\mathbb{N}$  the h-set $\mathcal{Q}_{4k}$, which is on $\Sigma_{0}$, is self $S$-symmetric. We emphasize that in the local coordinates given by $\psi_{0}$ the origin is a collision point. The larger the $k$ the closer $\mathcal{Q}_{4k}$ is to this collision. At the same time by Lemma \ref{lem:collision_avoidance_3} we know that $\mathcal{R}_{4k}$ and $\mathcal{Q}_{4k}$, for $k\in\mathbb{N}$, can not contain collisions.

Since we have many Poincar\'e sections, in the below statements and derivations we will add the information on which sections the corresponding h-sets are positioned by writing the sections below.  

\begin{lemma}
    \label{lem:R-Q-coverings}
    For $k\in\mathbb{N}$ we have the following coverings (here we use the short-hand notation (\ref{eq:sec-covering-1} --\ref{eq:sec-covering-2}))
\begin{equation}
\begin{array}[c]{ccccccccc}
\mathcal{R}_{4k} & \overset{P_{1}}{\implies} & \mathcal{R}_{4k+1}  & \overset{P_{2}}{\implies}& \mathcal{R}_{4k+2} & \overset{P_{3}}{\implies} & \mathcal{R}_{4k+3} & \overset{P_{0}}{\implies} & \mathcal{R}_{4\left(  k+1\right)  }\\
& & &  & &  & \mathcal{R}_{4k+3} & \overset{P_{0}}{\implies} & \mathcal{Q}_{4\left(  k+1\right)  }. \medskip\\
\Sigma_{0} & & \Sigma_{1}  & & \Sigma_{2}  &  & \Sigma_{3} &  & \Sigma_{0}%
\end{array}
\label{eq:R-Q-coverings}
\end{equation}

\end{lemma}

\begin{proof}
For $i=0,\ldots,3$ we define $g_{i}:\mathbb{R}^{2}\rightarrow
\mathbb{R}^{2}$ as%
\[
g_{0}:=\left(  \psi_{0}\right)  ^{-1}\circ P_{0}\circ\psi_{3}\qquad
\text{and\qquad}g_{i}:=\left(  \psi_{i}\right)  ^{-1}\circ P_{0}\circ
\psi_{i-1}\qquad\text{for }i=1,2,3.
\]
Note that due to (\ref{eq:points-on-orbit}) we have $g_{i}\left(  0\right)
=0$ for $i=0,\ldots,3$. We have validated with interval arithmetic computations that the
maps $g_{0},\ldots, g_{3}$ satisfy the assumptions of Lemma
\ref{lem:gk-approach}, which implies (\ref{eq:R-Q-coverings}). This concludes the proof.
\end{proof}
Below, in addition to the Poincar\'e maps for the symmetric sections introduced in (\ref{eq:PS-maps}) we consider also $P^S_0:\Sigma_0=S\Sigma_0 \to S\Sigma_{3}$.
\begin{corollary}
From Lemmas \ref{lem:S_backsymmetry} and \ref{lem:R-Q-coverings} it follows
that
\begin{equation}%
\begin{array}
[c]{rcccccccl}%
S\mathcal{R}_{4\left(  k+1\right)  } & \overset{P_{0}^{S}}{\Longleftarrow} &
S\mathcal{R}_{4k+3} & \overset{P_{3}^{S}}{\Longleftarrow} & S\mathcal{R}_{4k+2} & \overset{P_{2}^{S}}{\Longleftarrow} & S\mathcal{R}_{4k+1} &
\overset{P_{1}^{S}}{\Longleftarrow} & S\mathcal{R}_{4k}\\
S\mathcal{Q}_{4\left(  k+1\right)  } & \overset{P_{0}^{S}}{\Longleftarrow} & S\mathcal{R}_{4k+3} & & & & & & \medskip \\
S\Sigma_{0}=\Sigma_{0} &  & S\Sigma_{3} &  &S\Sigma_{2}   & &S\Sigma_{1}  &  & S\Sigma_{0}=\Sigma_{0}%
\end{array}
\label{eq:R-Q-back-sequence}%
\end{equation}
We emphasize that since $\mathcal{Q}_{4\left(  k+1\right)  }$ is self
$S$-symmetric, by Lemma \ref{lem:self-symmetric-covering} we also have
\[
\begin{array}
[c]{ccc}%
\mathcal{Q}_{4\left(  k+1\right)  } & \overset{P_{0}^{S}}{\Longleftarrow} &
S\mathcal{R}_{4k+3}.\\
\Sigma_{0} &  & S\Sigma_{3}%
\end{array}
\]
\end{corollary}

We also have the following covering relation:

\begin{lemma}
For the h-set $\mathcal{N}_{4}$ from (\ref{eq:covering_sequence_o1}), the following back-covering is valid
\begin{equation}
    \label{eq:gluing}
    \begin{array}[c]{rcc}
        S\mathcal{R}_{1} & \overset{P_{4}}{\Longleftarrow} & \mathcal{N}_{4}.\\
        S\Sigma_{1}=\Sigma_{3} &  & \Sigma_{4}%
    \end{array}
\end{equation}
\end{lemma}

\begin{proof}
    We have validated this covering by interval arithmetic computations using CAPD.
\end{proof}

\begin{corollary}
From (\ref{eq:covering_sequence_o1}), by taking a fragment of the sequence, we see that we have
\begin{equation}%
\begin{array}
[c]{ccccccc}%
\mathcal{N}_{4} & \overset{P_{5}}{\implies} & \mathcal{N}_{5} & \overset
{P_{6}}{\implies} & \ldots & \overset{P_{K}}{\implies} & \mathcal{N}_{K}.\\
\Sigma_{4} &  & \Sigma_{5} &  &  &  & \Sigma_{K}%
\end{array}
\label{eq:from-N4-seq}%
\end{equation}
We can join the sequences (\ref{eq:R-Q-back-sequence}) and (\ref{eq:from-N4-seq}), by placing (\ref{eq:gluing}) between them to glue them, to obtain (here we neglect to write out the Poincer\'e maps above the arrows to avoid clutter)
\begin{equation}%
\begin{array}
[c]{ccccccccccccccc}%
\mathcal{Q}_{4k} &\Leftarrow & S\mathcal{R}_{4k-1} & \Leftarrow & \ldots &
\Leftarrow & S\mathcal{R}_{1} & \Leftarrow & \mathcal{N}_{4} & \Rightarrow &
\mathcal{N}_{5} & \Rightarrow & \ldots & \Rightarrow & \mathcal{N}_{K}, \\ 
\Sigma_{0} &  & S\Sigma_{3} &  &  &  & \Sigma_{3} &  & \Sigma_{4}
&  & \Sigma_{5} &  &  &  & \Sigma_{K}
\end{array}
\label{eq:half-seq-1}%
\end{equation}
(recall that $S\Sigma_{1}=\Sigma_{3}$, which is why below $S\mathcal{R}_{1}$ we have $\Sigma_{3}$)
and by the $S$-symmetry of the system and since $N_{K}$ and $Q_{4k}$ are self $S$-symmetric, from (\ref{eq:half-seq-1}) we obtain
\begin{equation}%
\begin{array}
[c]{ccccccccccccccc}%
\mathcal{N}_{K} & \Leftarrow & S\mathcal{N}_{K-1} & \Leftarrow & \ldots &
\Leftarrow & S\mathcal{N}_{4} & \Rightarrow & \mathcal{R}_{1} & \Rightarrow &
\mathcal{R}_{2} & \Rightarrow & \ldots & \Rightarrow & \mathcal{Q}_{4k}.\\
\Sigma_{K} &  & S\Sigma_{K-1} &  &  &  & S\Sigma_{4} &  & \Sigma_{1} &  &
\Sigma_{2} &  &  &  & \Sigma_{0}%
\end{array}
\label{eq:half-seq-2}%
\end{equation}
\end{corollary}
By joining (\ref{eq:half-seq-1}) with (\ref{eq:half-seq-2}) we obtain the sequence%
\begin{equation}%
\begin{array}
[c]{ccccccccc}%
\mathcal{N}_{K} & \Leftarrow & \ldots & \Rightarrow & \mathcal{Q}_{4k} &
 \Leftarrow & \ldots & \Rightarrow & \mathcal{N}_{K}.\\
\Sigma_{K} &  &  &  & \Sigma_{0} &  &  &  & \Sigma_{K}%
\end{array}
\label{eq:full-sequence-long}%
\end{equation}
Instead of writing out the sequence (\ref{eq:full-sequence-long}) we will use the following short hand notation
\[
    \mathcal{N}_{K}\overset{k}{\longrightarrow}\mathcal{N}_{K}.
\]
We are ready to prove Theorem \ref{thm:primary_result}. 

\begin{proof}[Proof of Theorem \ref{thm:primary_result}]
From Lemmas \ref{lem:collision_avoidance_1} and \ref{lem:collision_avoidance_3} we know that trajectories that shadow the sequences of coverings $\mathcal{N}_{K}\overset{k}{\longrightarrow}\mathcal{N}_{K}$ (which is the short hand notation for (\ref{eq:full-sequence-long})) can not collide with Earth. Moreover, from our construction the larger the $k$, the closer is the set $\psi_0(\mathcal{Q}_{4k})$ to a collision with Earth.  

We are ready to prove (\ref{eq:main}) for $X,Y\in \{Oc,A\}$.

We obtain (\ref{eq:main}) for $X=Y=Oc$ from the following sequence of covering relations
\[
    \cdots \overset{3}{\longrightarrow}
    \mathcal{N}_K\overset{2}{\longrightarrow}
    \mathcal{N}_K\overset{1}{\longrightarrow}
    \mathcal{N}_K\overset{2}{\longrightarrow}
    \mathcal{N}_K\overset{3}{\longrightarrow}\cdots .
\]
This is because by Theorem \ref{thm:generic_covering_relations} we can find a trajectory which shadows this infinite sequence of coverings.

Similarly, when $X=Oc$ and $Y=A$, the motions (\ref{eq:main}) follow from
\[
    \cdots \overset{3}{\longrightarrow}
    \mathcal{N}_K\overset{2}{\longrightarrow}
    \mathcal{N}_K\overset{1}{\longrightarrow}
    \mathcal{N}_K\overset{1}{\longrightarrow}
    \mathcal{N}_K\overset{1}{\longrightarrow}\cdots,
\]
and for $X=A$ and $Y=Oc$ we can use
\[
    \cdots \overset{1}{\longrightarrow}
    \mathcal{N}_K\overset{1}{\longrightarrow}
    \mathcal{N}_K\overset{1}{\longrightarrow}
    \mathcal{N}_K\overset{2}{\longrightarrow}
    \mathcal{N}_K\overset{3}{\longrightarrow}\cdots.
\]
The result for $X=Y=A$ follows for instance from
\[
    \cdots \overset{1}{\longrightarrow}
    \mathcal{N}_K\overset{1}{\longrightarrow}
    \mathcal{N}_K\overset{1}{\longrightarrow}
    \mathcal{N}_K\overset{1}{\longrightarrow}
    \mathcal{N}_K\overset{1}{\longrightarrow}\cdots,
\]
though here the statement is trivial since any solution away from collision would suffice; for instance a libration fixed point of the PCR3BP.
 
We now move to the cases involving collisions. 

The case for $X=Y=C$ follows from the existence of the ejection/collision orbit established in Theorem \ref{thm:lyapunov_orbit_energy}.
    
To link collisions with other types of motions let us introduce the following notation for the sequences of covering relations (\ref{eq:covering_sequence_o1}) and (\ref{eq:covering_sequence_o2}). Instead of writing out (\ref{eq:covering_sequence_o1}) and (\ref{eq:covering_sequence_o2}) we shall write
\[
    \mathcal{N}_{0}\underset{c}{\longrightarrow}\mathcal{N}_{K}\qquad\mbox{and}\qquad \mathcal{N}_{K}\underset{c}{\longrightarrow}\mathcal{N}_{0},
\]
respectively. (Here `c' stands for `collision'.)
Let us also observe that by Lemma \ref{lem:collision-in-psi0} the collision curve when intersected with $\mathcal{N}_0$ in the local coordinates given by $\psi_0$, constitutes both a horizontal as well as a vertical disc (in the sense of Definitions \ref{def:horizontal-disc} and \ref{def:vertical-disc}). See Figure \ref{fig:overview2}: since the collision line is diagonal it is both a horizontal  as well as a vertical disc.

From Theorem \ref{thm:cover-discs} it follows that every horizontal disc contains a point, whose forward trajectory shadows an infinite sequence of covering relations, which starts from the h-set containing the horizontal disc. Also, from Theorem \ref{thm:cover-discs}, it follows that for every vertical disc there exists a point, whose backward trajectory shadows an infinite sequence of coverings, which ends at the h-set containing the vertical disc. Thus, the case $X=A$ and $Y=C$ follows from
 \[
    \cdots \overset{1}{\longrightarrow}
    \mathcal{N}_K\overset{1}{\longrightarrow}
    \mathcal{N}_K\overset{1}{\longrightarrow}
    \mathcal{N}_K\underset{c}{\longrightarrow}
    \mathcal{N}_0,
\]
and the case $X=C$ and $Y=A$ follows from
 \[
    \mathcal{N}_0\underset{c}{\longrightarrow}
    \mathcal{N}_K\overset{1}{\longrightarrow}
    \mathcal{N}_K\overset{1}{\longrightarrow}
    \mathcal{N} \overset{1}{\longrightarrow} \cdots.
\]
Similarly,
\[
	\cdots \overset{3}{\longrightarrow} 
	\mathcal{N}_K\overset{2}{\longrightarrow} 
	\mathcal{N}_K\overset{1}{\longrightarrow}
    \mathcal{N}_K\underset{c}{\longrightarrow}
    \mathcal{N}_0
\]
implies $X=Os$ and $Y=C$, and 
\[
    \mathcal{N}_0\underset{c}{\longrightarrow}
    \mathcal{N}_K\overset{1}{\longrightarrow}
    \mathcal{N}_K\overset{2}{\longrightarrow}
    \mathcal{N} \overset{3}{\longrightarrow} \cdots.
\]
implies $X=C$ and $Y=Os$.
 
We have established all the motions (\ref{eq:main}) from Theorem \ref{thm:primary_result}. 
 
We finish by observing that the periodic orbits from the final assertion of the theorem follow from the fact that we have
\[
    \mathcal{N}_K\overset{n}{\longrightarrow} \mathcal{N}_K
\]
for arbitrarily large $n\in \mathbb{N}$.
This concludes our proof.
\end{proof}

\section{Conclusions}
We have shown an explicit construction, which leads the connections between the forward and backward time motions exhibiting: a collision, oscillatory motion to collision, motion away from collision, in the Earth-Moon planar circular restricted three body problem. The construction was based on exploiting the homoclinic connections for the family of Lyapunov orbits, which persist in the Levi-Civita coordinates for the ejection/collision orbit. The construction is not dependent on the specific mass parameter of the Earth-Moon system. We use this system as an example, but the method is applicable to other mass parameters, provided that the Lyapunov family collides with one of the primaries.
We believe that our construction can be matched with the one from \cite{MR4391693}, which involves the hyperbolic, parabolic and oscillatory motions to infinity. The trajectories for the motions in \cite{MR4391693} are constructed so that they enter the collision domains of the primaries. We believe that the construction from the current paper and from \cite{MR4391693}  can be combined. This is the subject of ongoing work.

\section*{Acknowledgments}
\noindent We would like to thank Jason Mireles-James and Piotr Zgliczy\'nski for helpful discussions.

\appendix

\section{Appendix}

\subsection{Coordinate systems}
\label{appendix:coordinate_systems}

Let $\epsilon:= 8.5 \cdot 10^{-10}$. This coefficient will play the role of a rescaling factor for our coordinate changes. Its role is to have the h-sets $\mathcal{N}_0,\ldots, \mathcal{N}_{18}$ in the local coordinates given by $\psi_0,\ldots, \psi_{18}$, respectively, to be of the same size and of the form $N_c=[-1,1]^2\subset \mathbb{R}^2$. In short, the h-sets in the local coordinates will be of order one, but their size in the coordinates of the PCR3BP will be of order $10^{-10}$.
 
Let $\mathcal{U}$ denote the vector normalization i.e. $\mathcal{U}(w) = w / \| w \|$.  We choose the matrices $A_k\in\mathbb{R}^{4\times4}$ for the definition of the coordinate changes $\psi_k$ (see (\ref{eq:A-form}), (\ref{eq:Lambda-def}) and (\ref{eq:local-parametrization-choice})) to be
\begin{equation}
    A_k:= \, \left[
        \epsilon\,\hat{u}_k \quad
        \epsilon\,\hat{s}_k \quad
        \epsilon\,\mathcal{U} \left( J\nabla \Gamma(w_k) \right) \quad
        \epsilon\,\mathcal{U} \left( \nabla \Gamma(w_k) \right) \right], \qquad \mbox{for }k=1,\ldots,18, \label{eq:Ak-choice}
\end{equation}
where the vectors $w_k$, $\hat{u}_k$ and $\hat{s}_k$ are chosen as written out in Tables \ref{table:wk}, \ref{table:uk} and \ref{table:sk}, respectively. Note that the choice of $A_k$ in (\ref{eq:Ak-choice}) is of the form (\ref{eq:A-form}); with a rescaling of the columns. We use the matrices $A_k$ from (\ref{eq:Ak-choice}) to define $\psi_k$ using (\ref{eq:local-parametrization-choice}), for $k=1,\ldots,18$.

\begin{table}[H]
    \centering
    \scriptsize
    \begin{tabular}{lllll}
        \hline
         & $u$ & $\phantom{-}v$ & $\phantom{-}p_u$ & $\phantom{-}p_v$ \\
        \hline
        $w_{0}$ & $\phantom{-}0$ & $\phantom{-}0$ & $\phantom{-}0$ & $\phantom{-}2.81112771399$ \\
        $w_{1}$ & $\phantom{-}1.01939532911$ & $\phantom{-}0.0352455355111$ & $-1.57992203844$ & $-0.220188078268$ \\
        $w_{2}$ & $\phantom{-}0.997511358555$ & $\phantom{-}0$ & $\phantom{-}0$ & $-3.10035066234$ \\
        $w_{3}$ & $\phantom{-}1.01939532911$ & $-0.0352455355111$ & $\phantom{-}1.57992203844$ & $-0.220188078268$ \\
        $w_{4}$ & $\phantom{-}1.82311329298$e-10 & $\phantom{-}0$ & $\phantom{-}3.0389849377$e-09 & $\phantom{-}2.81112771399$ \\
        $w_{5}$ & $\phantom{-}1.01939364734$ & $\phantom{-}0.0352429461228$ & $-1.57993713742$ & $-0.220199770725$ \\
        $w_{6}$ & $\phantom{-}0.997511359537$ & $\phantom{-}0$ & $\phantom{-}3.81728719867$e-07 & $-3.1003514216$ \\
        $w_{7}$ & $\phantom{-}1.01939365573$ & $-0.0352429413345$ & $\phantom{-}1.57993769304$ & $-0.22019938229$ \\
        $w_{8}$ & $\phantom{-}1.23284045903$e-07 & $\phantom{-}0$ & $\phantom{-}2.05626985632$e-06 & $\phantom{-}2.81112771099$ \\
        $w_{9}$ & $\phantom{-}1.01939456403$ & $\phantom{-}0.0352429632775$ & $-1.57993456608$ & $-0.220196017823$ \\
        $w_{10}$ & $\phantom{-}0.997512022626$ & $\phantom{-}0$ & $\phantom{-}0.000258254271245$ & $-3.10086432733$ \\
        $w_{11}$ & $\phantom{-}1.01940023627$ & $-0.0352397234978$ & $\phantom{-}1.58031042535$ & $-0.219933213027$ \\
        $w_{12}$ & $\phantom{-}8.34102256631$e-05 & $\phantom{-}0$ & $\phantom{-}0.00139148280123$ & $\phantom{-}2.8111253283$ \\
        $w_{13}$ & $\phantom{-}1.02000949836$ & $\phantom{-}0.0352470440096$ & $-1.57821685002$ & $-0.21769836291$ \\
        $w_{14}$ & $\phantom{-}0.997948108217$ & $\phantom{-}0$ & $\phantom{-}0.196225509832$ & $-3.49133960833$ \\
        $w_{15}$ & $\phantom{-}1.023984317$ & $-0.0326239190489$ & $\phantom{-}1.83944484044$ & $-0.0123000584495$ \\
        $w_{16}$ & $\phantom{-}0.0828113236996$ & $-1.13042936454$ & $\phantom{-}0.0931363222648$ & $\phantom{-}1.32036952924$ \\
        $w_{17}$ & $\phantom{-}0.0742904549562$ & $\phantom{-}0$ & $\phantom{-}1.24626303185$ & $\phantom{-}2.51265379059$ \\
        $w_{18}$ & $\phantom{-}1.26583072872$ & $\phantom{-}0$ & $\phantom{-}0$ & $\phantom{-}0.120135068527$ \\
    \end{tabular}
    \caption{Approximate values of the components of the vectors $w_k$.\label{table:wk}}
\end{table}

For the coordinate change $\psi_0$ we choose the constants $d_1$ and $d_2$ from (\ref{eq:psi_0-def}) 
so that the map $\psi_0$ is well aligned with the system dynamics; we require that:
\begin{equation}
    \label{eq:why-d1-d2}
    \frac{\partial \psi_0}{\partial z_1} (0,0) \approx \epsilon\,\hat{u}_0, \quad
    \frac{\partial \psi_0}{\partial z_2} (0,0) \approx \epsilon\,\hat{s}_0,
\end{equation}
and the choice which leads to this is
\begin{equation}
    \label{eq:d1-d2-choice}
    d_1= \epsilon\cdot0.0598649594810129 \qquad \mbox{and}\qquad
    d_2 = \epsilon\cdot0.997908614890024. 
\end{equation}

\begin{table}[H]
    \scriptsize
    \centering
    \begin{tabular}{lllll}
        \hline
         & $\phantom{-}u$ & $\phantom{-}v$ & $\phantom{-}p_u$ & $\phantom{-}p_v$ \\
        \hline
        $\hat{u}_{0}$ & $\phantom{-}0.0598653990819$ & $\phantom{-}0.0243392345095$ & $\phantom{-}0.99790861489$ & $-0.00146013168443$ \\
        $\hat{u}_{1}$ & $\phantom{-}0.0512722628998$ & $-0.0539645641977$ & $-0.0790749932687$ & $\phantom{-}0.106664891545$ \\
        $\hat{u}_{2}$ & $\phantom{-}0.0123990872795$ & $\phantom{-}0.00623386910969$ & $\phantom{-}8.06085056107$e-06 & $-9.58885003297$ \\
        $\hat{u}_{3}$ & $\phantom{-}0.0655510635754$ & $-0.0519328405322$ & $\phantom{-}0.994851864788$ & $\phantom{-}1.26355518097$ \\
        $\hat{u}_{4}$ & $\phantom{-}0.0598653990819$ & $\phantom{-}0.0243392345095$ & $\phantom{-}0.99790861489$ & $-0.00146013168443$ \\
        $\hat{u}_{5}$ & $\phantom{-}0.0625757321266$ & $-0.0658560755461$ & $-0.0965279955578$ & $\phantom{-}0.13018905303$ \\
        $\hat{u}_{6}$ & $\phantom{-}0.0184674816754$ & $\phantom{-}0.00928486590201$ & $\phantom{-}1.20031980776$e-05 & $-14.2818597745$ \\
        $\hat{u}_{7}$ & $\phantom{-}0.119128113811$ & $-0.0943647041922$ & $\phantom{-}1.80852788038$ & $\phantom{-}2.2967467384$ \\
        $\hat{u}_{8}$ & $\phantom{-}0.132804359411$ & $\phantom{-}0.0539941945251$ & $\phantom{-}2.21374370854$ & $-0.0032407835347$ \\
        $\hat{u}_{9}$ & $\phantom{-}0.138815104952$ & $-0.146096387381$ & $-0.214131083267$ & $\phantom{-}0.288817842031$ \\
        $\hat{u}_{10}$ & $\phantom{-}0.0409636350452$ & $\phantom{-}0.0205951178031$ & $\phantom{-}2.24064629998$e-05 & $-31.6930288739$ \\
        $\hat{u}_{11}$ & $\phantom{-}0.264355726553$ & $-0.209336295197$ & $\phantom{-}4.01213882195$ & $\phantom{-}5.0966647523$ \\
        $\hat{u}_{12}$ & $\phantom{-}0.294792860317$ & $\phantom{-}0.120543706008$ & $\phantom{-}4.91486295091$ & $-0.00967322389548$ \\
        $\hat{u}_{13}$ & $\phantom{-}0.305475938325$ & $-0.327843230549$ & $-0.468071313644$ & $\phantom{-}0.654208911006$ \\
        $\hat{u}_{14}$ & $\phantom{-}0.0842120960402$ & $\phantom{-}0.0422016823499$ & $-0.00607062017782$ & $-89.0241271111$ \\
        $\hat{u}_{15}$ & $\phantom{-}0.731763394696$ & $-0.458286780463$ & $\phantom{-}8.89272504938$ & $\phantom{-}13.9801270312$ \\
        $\hat{u}_{16}$ & $-3.39472774076$ & $-2.60519046749$ & $\phantom{-}6.3830932378$ & $-4.93809685454$ \\
        $\hat{u}_{17}$ & $\phantom{-}0.504441646028$ & $\phantom{-}0.387143559261$ & $\phantom{-}4.32240508069$ & $-2.22878232827$ \\
        $\hat{u}_{18}$ & $\phantom{-}0.0497937735344$ & $-0.534053889757$ & $\phantom{-}0.326462982573$ & $\phantom{-}0.0814565811784$ \\
        \hline
    \end{tabular}
    \caption{Approximate values of the components of the vectors $\hat{u}_k$.\label{table:uk}}
\end{table}

\begin{table}[H]
	\scriptsize
    \centering
    \begin{tabular}{lllll}
        \hline
         & $\phantom{-}u$ & $\phantom{-}v$ & $\phantom{-}p_u$ & $\phantom{-}p_v$ \\
        \hline
        $\hat{s}_{0}$ & $\phantom{-}0.0598653990819$ & $-0.0243392345095$ & $-0.99790861489$ & $-0.00146013168443$ \\
        $\hat{s}_{1}$ & $\phantom{-}0.0655510635754$ & $\phantom{-}0.0519328405322$ & $-0.994851864788$ & $\phantom{-}1.26355518097$ \\
        $\hat{s}_{2}$ & $\phantom{-}0.0123990872795$ & $-0.00623386910969$ & $-8.06085056107$e-06 & $-9.58885003297$ \\
        $\hat{s}_{3}$ & $\phantom{-}0.0512722628998$ & $\phantom{-}0.0539645641977$ & $\phantom{-}0.0790749932687$ & $\phantom{-}0.106664891545$ \\
        $\hat{s}_{4}$ & $\phantom{-}0.0598653990819$ & $-0.0243392345095$ & $-0.99790861489$ & $-0.00146013168443$ \\
        $\hat{s}_{5}$ & $\phantom{-}0.0536845559625$ & $\phantom{-}0.042525059308$ & $-0.815005301074$ & $\phantom{-}1.03501866893$ \\
        $\hat{s}_{6}$ & $\phantom{-}0.00832229054357$ & $-0.00418418326664$ & $-5.41171804734$e-06 & $-6.4360581837$ \\
        $\hat{s}_{7}$ & $\phantom{-}0.0281994747078$ & $\phantom{-}0.0296777546137$ & $\phantom{-}0.0434999318962$ & $\phantom{-}0.058669162844$ \\
        $\hat{s}_{8}$ & $\phantom{-}0.026978030238$ & $-0.0109685022423$ & $-0.449703262016$ & $-0.000657678209258$ \\
        $\hat{s}_{9}$ & $\phantom{-}0.0242011480547$ & $\phantom{-}0.0191715734179$ & $-0.367373745058$ & $\phantom{-}0.466569759349$ \\
        $\hat{s}_{10}$ & $\phantom{-}0.0037506919161$ & $-0.00188532541287$ & $-2.82255343343$e-06 & $-2.9018616866$ \\
        $\hat{s}_{11}$ & $\phantom{-}0.0127077970338$ & $\phantom{-}0.0133767067373$ & $\phantom{-}0.0196090894894$ & $\phantom{-}0.0264540615654$ \\
        $\hat{s}_{12}$ & $\phantom{-}0.0121439073914$ & $-0.0049838504505$ & $-0.202713054229$ & $-0.000197805387775$ \\
        $\hat{s}_{13}$ & $\phantom{-}0.0112900526678$ & $\phantom{-}0.00929003414053$ & $-0.161456580967$ & $\phantom{-}0.211692282532$ \\
        $\hat{s}_{14}$ & $\phantom{-}0.00144550242937$ & $-0.000618072920234$ & $-0.000105478869183$ & $-1.52819846002$ \\
        $\hat{s}_{15}$ & $\phantom{-}0.00455638237545$ & $\phantom{-}0.00563436828252$ & $\phantom{-}0.00813263074882$ & $\phantom{-}0.0134836469778$ \\
        $\hat{s}_{16}$ & $\phantom{-}0.0299785746856$ & $\phantom{-}0.00891662895616$ & $-0.0343147668569$ & $\phantom{-}0.0340399244557$ \\
        $\hat{s}_{17}$ & $-0.026049583212$ & $-0.0475477618468$ & $-0.409735409863$ & $\phantom{-}0.20805798649$ \\
        $\hat{s}_{18}$ & $\phantom{-}0.0497937735344$ & $\phantom{-}0.534053889757$ & $-0.326462982573$ & $\phantom{-}0.0814565811784$ \\
        \hline
    \end{tabular}
    \caption{Approximate values of the components of the vectors $\hat{s}_k$.}
    \label{table:sk}
\end{table}

\bibliographystyle{plain} 
\bibliography{references}

\end{document}